\documentclass[12pt]{amsart}
\usepackage{geometry}
\usepackage{amsmath}
\usepackage{verbatim}
\usepackage{tikz}
\usepackage{pgfplots}
\newtheorem{proposition}{Proposition}

\newtheorem{theorem}{Theorem}
\newtheorem{lemma}{Lemma}

\newcommand{\Mod}[1]{\, (\mathrm{mod}\, #1)}

\newcommand{\Addresses}{{
  \footnotesize
  \bigskip
  \footnotesize

	\textsc{Mathematical Institute, University of Oxford, Oxford, OX2 6GG, UK}\par\nopagebreak
	\textit{E-mail address:} \texttt{ofir.goro@gmail.com}

	\medskip

	\textsc{Department of Mathematics and Statistics, University of Turku, 20014 Turku, Finland}\par\nopagebreak
	\textit{E-mail address:} \texttt{ksmato@utu.fi}

	\medskip

	\textsc{Department of Mathematics, Caltech, 1200 E California Blvd, Pasadena, CA,91125}\par\nopagebreak
	\textit{E-mail address:} \texttt{maksym.radziwill@gmail.com}

	\medskip

	\textsc{Department of Mathematics and Statistics, Queen's University, Kingston, Ontario, K7L 3N6, Canada}\par\nopagebreak
	\textit{E-mail address:} \texttt{brad.rodgers@queensu.ca}

}}

\begin{document}

\title[Variance of squarefree integers]{On the variance of squarefree integers in short intervals and arithmetic progressions}

\author{Ofir Gorodetsky, Kaisa Matom\"aki, Maksym Radziwi\l\l, Brad Rodgers}

\begin{abstract}
  We evaluate asymptotically the variance of the number of squarefree integers up to $x$ in short intervals of length $H < x^{6/11 - \varepsilon}$ and the variance of the number of squarefree integers up to $x$ in arithmetic progressions modulo $q$ with $q > x^{5/11 + \varepsilon}$. On the assumption of respectively the Lindel\"of Hypothesis and the Generalized Lindel\"of Hypothesis we show that these ranges can be improved to respectively $H < x^{2/3 - \varepsilon}$ and $q > x^{1/3 + \varepsilon}$. Furthermore we show that obtaining a bound sharp up to factors of $H^{\varepsilon}$ in the full range $H < x^{1 - \varepsilon}$ is equivalent to the Riemann Hypothesis. These results improve on a result of Hall (1982) for short intervals, and earlier results of Warlimont, Vaughan, Blomer, Nunes and Le Boudec in the case of arithmetic progressions.
  \end{abstract}

  \maketitle

\section{Introduction}

\subsection{Main results}

An integer $n \geq 1$ is squarefree if it is not divisible by the square of a prime.
By analogy with questions about prime numbers,
a basic problem in analytic number theory is to understand the distribution of squarefree numbers
in arithmetic progressions and in short intervals. Squarefree numbers ought to be a simpler, more
regular sequence than primes, and yet they present distinct challenges; for instance
we can determine whether $n$ is prime in polynomial time \cite{primesP},
but there is no known polynomial time algorithm to determine whether $n$ is squarefree. 

It was conjectured by Montgomery (see \cite{Croft}) that for any given $\varepsilon \in (0, 1/100)$, and $(a,q) = 1$,
\begin{equation} \label{eq:montconj}
\sum_{\substack{n \leq x \\ n \equiv a \Mod{q}}} \mu^2(n) = \frac{6}{\pi^2} \cdot \frac{x}{q} \prod_{p | q} \Big ( 1 - \frac{1}{p^2} \Big ) + O_{\varepsilon} \Big ( (x / q)^{1/4 + \varepsilon} \Big ).
\end{equation}
uniformly in $1 \leq q \leq x^{1 - \varepsilon}$. 
This conjecture is difficult for two reasons. In the regime of large $q$ of size roughly $x^{1 - \varepsilon}$ the left-hand side contains only $x^{\varepsilon}$ terms and even establishing an asymptotic is open\footnote{In fact establishing that the left-hand side is positive for $q = x^{1 - \varepsilon}$ is open!} (see the work of Nunes \cite{nunes2017} for the best result in this direction). In the regime of small $q$ of size about $x^{\varepsilon}$ establishing an asymptotic is easy but obtaining an error term as good as $O_{\varepsilon}((x / q)^{1/4 + \varepsilon})$ is an open problem, even conditionally on the Generalized Riemann Hypothesis.

Analogously we conjecture that for any given $\varepsilon \in (0, \tfrac{1}{100})$, uniformly in $x^{\varepsilon} \leq H \leq x$, 
\begin{equation} \label{eq:shortconj}
\sum_{\substack{x < n \leq x + H}} \mu^2(n) = \frac{6 H}{\pi^2} + O_{\varepsilon}(H^{1/4 + \varepsilon}).
\end{equation}
Similarly to the case of arithmetic progressions, when $H$ is close to $x^{\varepsilon}$ no asymptotic estimates are known (see the work of Tolev \cite{tolev2006} and Filaseta-Trifonov \cite{filaseta1992} for the best unconditional results in this direction and \cite[Thm.~A.1]{carmon2019}, \cite{granville1998} for results conditional on the ABC conjecture). Meanwhile for large $H$, say $H = x$, estimating \eqref{eq:shortconj} asymptotically is straightforward, but obtaining an error term $O_{\varepsilon}(x^{1/4 + \varepsilon})$ is an open problem, even conditionally on the Riemann Hypothesis (see \cite{Liu2016} for the best result in this direction).

An important feature of both conjectures \eqref{eq:montconj} and \eqref{eq:shortconj} is that the error term is significantly smaller than the square-root of the number of terms being summed, in contrast to what a naive probabilistic model predicts.  

The conjectures \eqref{eq:montconj} and \eqref{eq:shortconj} imply the Riemann Hypothesis, and they are almost certainly deeper than the Riemann Hypothesis. 
Nonetheless one can still hope
to investigate them on average over residue classes for \eqref{eq:montconj} or on average over short intervals for \eqref{eq:shortconj}.
Importantly, establishing \eqref{eq:montconj} on average is easier when $q$ is large than when $q$ is small, since a large $q$ allows for more averaging over the residue classes $a \Mod{q}$. Similarly establishing \eqref{eq:shortconj} on average is easier when $H$ is small, since there are more non-overlapping short intervals $[x, x + H]$ to average over compared to the case when $H$ is large. In fact when there is little averaging (i.e $q$ small or $H$ large), the averaged versions of \eqref{eq:montconj} and \eqref{eq:shortconj} are not significantly easier than the non-averaged version, see Theorem \ref{thm:main_upper} for a concrete manifestation of this. 

In our first result we compute the variance of \eqref{eq:shortconj} on average over short intervals. We estimate the variance asymptotically thus making (on average) the error term in \eqref{eq:shortconj} more precise. 

\begin{theorem} \label{thm:main}
  Let $\varepsilon \in (0, \tfrac{1}{100})$ be given.
  Let $X \geq 1$ and $1 \le H \leq X^{6/11 - \varepsilon}$. Then 
  \begin{equation} \label{eq:main}
  \frac{1}{X} \int_{X}^{2X} \Big | \sum_{x < m \leq x + H} \mu^2(m) - \frac{6H}{\pi^2} \Big |^2 dx = C \sqrt{H} + O_{\varepsilon}(H^{1/2-\varepsilon/16})
  \end{equation}
  with 
  \begin{equation}\label{eq:C_dfn}
  C:= \frac{\zeta(3/2)}{\pi} \prod_p \Big(1 - \frac{3}{p^2} + \frac{2}{p^3}\Big).
  \end{equation}
  Assuming the Lindel\"of Hypothesis \eqref{eq:main} holds in the wider range $H \leq X^{2/3-\varepsilon}$.
\end{theorem}

We recall that the Lindel\"of Hypothesis follows from the Riemann Hypothesis and
asserts that for any given $\varepsilon > 0$ we have $|\zeta(\tfrac 12 + it)| \ll_{\varepsilon} 1 + |t|^{\varepsilon}$ for all $t \in \mathbb{R}$. In Theorem~\ref{thm:main_upper} we will show that if we had \eqref{eq:main} in the full range $H \leq X^{1 - \varepsilon}$ then the Riemann Hypothesis would ensue.

Theorem \ref{thm:main} extends a theorem of Hall \cite{hall1982} who showed that the asymptotic formula \eqref{eq:main} holds in the range $H \leq X^{2/9-\varepsilon}$. We will now explain why the range $H = X^{1/2}$ can be considered a threshold in this problem. It is reasonable to conjecture that given $\varepsilon > 0$,
for any $1 \leq h \leq x^{1 - \varepsilon}$,
\begin{equation} \label{eq:corsums}
\sum_{n \leq x} \mu^2(n) \mu^2(n + h) - C(h) x = O_{\varepsilon}(x^{1/4 + \varepsilon})
\end{equation}
with $C(h)$ a constant depending only on $h$. Summing this conjectural estimate over $h$ recovers Theorem \ref{thm:main} but only in the range $H < X^{1/2 - \varepsilon}$.
Thus Theorem \ref{thm:main} exploits (unconditionally!) additional cancellations between the error terms in \eqref{eq:corsums}.\footnote{Using estimates of Tsang \cite{tsang1985} for \eqref{eq:corsums} recovers Theorem \ref{thm:main} in the range $H \leq X^{8/33 - \varepsilon}$.} 

We now describe the analogue of Theorem \ref{thm:main} for the distribution of squarefree numbers in arithmetic progressions with a
given modulus. In this case for a given modulus $q$ the parameter $x / q$ has the same role as the length $H$ of the short interval
in Theorem \ref{thm:main}. While the results are analogous they are harder to prove, as is often the case with $q$-analogues. 

\begin{theorem} \label{thm:main2}
  Let $\varepsilon \in (0, \tfrac{1}{100})$ be given.
  Let $x \ge q \geq x^{5/11 + \varepsilon}$ be a prime. Then
  \begin{equation}
  \begin{aligned}\label{eq:main2}
  \frac{1}{\varphi(q)}\sum_{(a,q) = 1} & \Big | \sum_{\substack{m \leq x \\ m \,\equiv\, a \Mod{q}}} \mu^2(m) - \frac{6}{\pi^2}\cdot \frac{x}{q} \prod_{p | q} \Big(1-\frac{1}{p^2}\Big)^{-1} \Big |^2 \\  & = C \prod_{p | q} \Big (1 + \frac{2}{p} \Big )^{-1} \cdot \sqrt{\frac{x}{q}} + O_{\varepsilon}((x/q)^{1/2-\varepsilon/16}),
  \end{aligned}
  \end{equation}
  where $C$ is the same constant as in Theorem \ref{thm:main}. Assuming the Generalized Lindel\"of Hypothesis 
  the claim holds in the wider range $q > x^{1/3 + 30 \varepsilon}$. 
\end{theorem}

We recall that the Generalized Lindel\"of Hypothesis follows from the Generalized Riemann Hypothesis and asserts that for any given $\varepsilon > 0$ we have $|L(\tfrac 12 + it, \chi)| \ll_{\varepsilon} 1 + (|q| + |t|)^{\varepsilon}$ for all $t \in \mathbb{R}$ and all characters $\chi \Mod{q}$. 

For simplicity we have assumed in Theorem \ref{thm:main2} that $q$ is prime, but our methods are amenable to handling the general case of composite $q$ with a bit more effort.

Once extended to composite $q$ our
Theorem \ref{thm:main2} improves on results by Warlimont \cite{warlimont1980squarefree} and Vaughan \cite{vaughan2005variance} who obtain an asymptotic formula with an additional averaging over $q \leq Q$ in the range $x^{2/3} \leq Q = o(x)$.
Moreover, for prime values of $q$, Theorem \ref{thm:main2} improves on a succession of results by Blomer \cite{blomer2008}, Nunes \cite{nunes2015} (see also \cite{parry2019variance}) and Le Boudec \cite{leboudec2018} who considered individual averages over $(a,q) = 1$ as we do in Theorem \ref{thm:main2}. In particular Nunes showed that \eqref{eq:main2} holds in the range $x^{31/41+\varepsilon} \leq q = o(x)$ and Le Boudec showed that the left-hand side of \eqref{eq:main2} is $O_\varepsilon((x/q)^{1/2+\varepsilon})$ for all $\varepsilon > 0$ in the range $x^{1/2} \leq q \leq x$.

Keating and Rudnick \cite{keating2016} obtained Theorem \ref{thm:main} and Theorem \ref{thm:main2} in the context of function fields in the limit of a large field size. Their results hold in the (analogues of) the ranges $X^{\varepsilon} \leq H \leq X^{1 - \varepsilon}$ and $x^{\varepsilon} \leq q \leq x^{1 - \varepsilon}$. 
Our proofs of Theorem \ref{thm:main} and Theorem \ref{thm:main2} can be adapted in the setting of a fixed base field and large degree limit. In fact our proofs of Theorems \ref{thm:main} and \ref{thm:main2} were originally motivated by analogies with the function field setting. Since we ended up obtaining equally strong results in the setting of number fields we do not include the proofs in the function field setting.

Finally the next result shows that obtaining nearly optimal upper bounds for \eqref{eq:main} in a complete range is equivalent to the Riemann Hypothesis. 

\begin{theorem}\label{thm:main_upper}
  The Riemann Hypothesis holds if and only if for every $\varepsilon \in (0, \tfrac{1}{100})$ and every $1 \leq H \leq X^{1 - \varepsilon}$,
  \begin{equation} \label{eq:main_upper}
  \frac{1}{X} \int_{X}^{2X} \Big | \sum_{x < m \leq x + H} \mu^2(m) - \frac{6 H}{\pi^2} \Big |^2 dx \ll_{\varepsilon, \delta} H^{1/2 + \delta}
  \end{equation}
  for every $\delta > 0$.
\end{theorem}

Following the proof of Theorem \ref{thm:main_upper} one can show that for any smooth compactly supported $\Phi$, conditionally on the Generalized Riemann Hypothesis
\begin{equation} \label{eq:mm}
  \frac{1}{\varphi(q)} \sum_{(a,q) = 1} \Big | \sum_{\substack{m \equiv a \pmod{q}}} \mu^2(n) \Phi\Big ( \frac{n}{x} \Big ) - \frac{6}{\pi^2 \varphi(q)} \sum_{\substack{(m,q) = 1}} \Phi \Big ( \frac{m}{x} \Big ) \Big |^2 \ll_{\varepsilon, \delta} (x / q)^{1/2 + \delta}
\end{equation}
for all $\delta > 0$ and uniformly in $1 \leq q \leq x^{1 - \varepsilon}$ for any given $\varepsilon \in (0, \tfrac{1}{100})$. However it is not clear whether \eqref{eq:mm} implies the Generalized Riemann Hypothesis. Moreover replacing the smoothing $\Phi$ by sharp cut-offs appears to be difficult. For these reasons we decided not to pursue this further in the present paper. 

Finally, we note that we have made no effort to optimize the exponents of error terms $O_{\varepsilon}(H^{1/2-\varepsilon/16})$ and $O_{\varepsilon}((x/q)^{1/2-\varepsilon/16})$ in Theorems \ref{thm:main} and \ref{thm:main2}. Better power saving estimates, in more restricted ranges, can be found in the papers \cite{hall1982} and \cite{nunes2015}.

\begin{figure}[b]
  \centering
\begin{tikzpicture} 
\draw [blue, xshift=0cm] plot [smooth, tension=1] coordinates { (0.0,0.0) 
  (0.01,-0.0477276) (0.02,-0.29111) (0.03,-0.16764) (0.04,-0.203139) (0.05,-0.22641) (0.06,-0.212995) (0.07,-0.309636) (0.08,-0.369593) (0.09,-0.331721) (0.1,-0.416134) (0.11,-0.353806) (0.12,-0.340391) (0.13,-0.0824074) (0.14,-0.0445359) (0.15,0.188991) (0.16,-0.15222) (0.17,0.0690786) (0.18,-0.0520197) (0.19,-0.258717) (0.2,-0.0496474) (0.21,-0.378629) (0.22,-0.120646) (0.23,-0.0950027) (0.24,-0.216101) (0.25,-0.055945) (0.26,-0.225957) (0.27,-0.0780296) (0.28,-0.064615) (0.29,0.0588556) (0.3,0.0355849) (0.31,0.305797) (0.32,0.245841) (0.33,0.136971) (0.34,0.0403296) (0.35,0.0292873) (0.36,0.0549303) (0.37,0.251772) (0.38,-0.016068) (0.39,0.119631) (0.4,-0.00146724) (0.41,0.195374) (0.42,0.0131335) (0.43,0.0754619) (0.44,0.0766481) (0.45,0.0411489) (0.46,-0.0065787) (0.47,-0.127677) (0.48,-0.114262) (0.49,0.143721) (0.5,0.181593) (0.51,0.402891) (0.52,0.660874) (0.53,0.454177) (0.54,0.284165) (0.55,0.236437) (0.56,0.0664252) (0.57,0.397779) (0.58,0.239996) (0.59,0.351238) (0.6,0.413566) (0.61,0.414752) (0.62,0.232512) (0.63,0.0258145) (0.64,0.149285) (0.65,0.260527) (0.66,0.335084) (0.67,0.299585) (0.68,-0.0416253) (0.69,0.289729) (0.7,0.217544) (0.71,0.0597607) (0.72,0.0976322) (0.73,0.123275) (0.74,0.246746) (0.75,0.309074) (0.76,0.151291) (0.77,0.0179639) (0.78,0.165891) (0.79,0.167078) (0.8,0.168264) (0.81,0.206135) (0.82,0.0483516) (0.83,0.330792) (0.84,0.050724) (0.85,0.44322) (0.86,0.358807) (0.87,0.298851) (0.88,0.238895) (0.89,0.582478) (0.9,0.412466) (0.91,0.499251) (0.92,0.573808) (0.93,0.574994) (0.94,0.686236) (0.95,0.565138) (0.96,0.798665) (0.97,0.922135) (0.98,0.923321) (0.99,0.802223) (1,0.55884) (1.01,0.388828) (1.02,0.353329) (1.03,0.378972) (1.04,0.123361) (1.05,0.332431) (1.06,0.235789) (1.07,-0.0687356) (1.08,0.0669634) (1.09,0.129292) (1.1,0.020422) (1.11,-0.100676) (1.12,-0.0750333) (1.13,0.0484374) (1.14,0.0129382) (1.15,0.087495) (1.16,0.149823) (1.17,0.322208) (1.18,0.27448) (1.19,0.300123) (1.2,0.276853) (1.21,0.143526) (1.22,0.0835697) (1.23,0.341553) (1.24,0.134856) (1.25,-0.0106993) (1.26,0.161685) (1.27,0.150643) (1.28,-0.0560546) (1.29,0.0918729) (1.3,0.39877) (1.31,0.375499) (1.32,0.26663) (1.33,0.23113) (1.34,0.171174) (1.35,0.196817) (1.36,0.30806) (1.37,0.394845) (1.38,0.530544) (1.39,0.495045) (1.4,0.325033) (1.41,0.0571931) (1.42,0.241806) (1.43,0.0595654) (1.44,0.0974369) (1.45,-0.011433) (1.46,-0.071389) (1.47,-0.0824313) (1.48,0.0777246) (1.49,-0.15343) (1.5,0.129011) (1.51,0.10574) (1.52,0.253668) (1.53,0.193711) (1.54,0.207126) (1.55,0.269454) (1.56,0.0627571) (1.57,0.161771) (1.58,0.0284441) (1.59,-0.00705512) (1.6,-0.164839) (1.61,-0.065825) (1.62,-0.040182) (1.63,0.0588318) (1.64,-0.11118) (1.65,-0.476848) (1.66,-0.341148) (1.67,-0.229906) (1.68,-0.461061) (1.69,-0.239762) (1.7,-0.0184642) (1.71,-0.139563) (1.72,-0.0405488) (1.73,-0.0515911) (1.74,0.0474227) (1.75,-0.281559) (1.76,0.000881214) (1.77,0.124352) (1.78,0.162223) (1.79,0.175638) (1.8,0.176824) (1.81,0.288066) (1.82,0.105826) (1.83,0.302667) (1.84,0.0959696) (1.85,0.439552) (1.86,0.147256) (1.87,0.16067) (1.88,0.112943) (1.89,0.211956) (1.9,0.286513) (1.91,-0.0180117) (1.92,-0.00459706) (1.93,-0.199066) (1.94,-0.148966) (1.95,-0.0988661) (1.96,-0.0609946) (1.97,-0.157636) (1.98,0.112576) (1.99,0.028163) (2,0.0171207) (2.01,-0.0428354) (2.02,-0.0294208) (2.03,0.106278) (2.04,-0.0637338) (2.05,0.182021) (2.06,0.0242377) (2.07,0.0254238) (2.08,0.295636) (2.09,0.0400246) (2.1,-0.0933021) (2.11,-0.3) (2.12,-0.372184) (2.13,-0.260942) (2.14,-0.137471) (2.15,0.05937) (2.16,0.0727846) (2.17,0.0617424) (2.18,0.148528) (2.19,0.247541) (2.2,0.260956) (2.21,0.00534479) (2.22,0.0187594) (2.23,-0.0289682) (2.24,0.155645) (2.25,-0.222251) (2.26,-0.111009) (2.27,-0.109823) (2.28,-0.145322) (2.29,0.0759763) (2.3,-0.106264) (2.31,-0.251819) (2.32,-0.189491) (2.33,-0.22499) (2.34,-0.0159204) (2.35,0.0341795) (2.36,-0.0502334) (2.37,-0.171332) (2.38,0.135566) (2.39,0.381321) (2.4,0.125709) (2.41,0.273637) (2.42,0.262595) (2.43,0.410522) (2.44,0.130454) (2.45,0.180554) (2.46,-0.00168645) (2.47,0.268526) (2.48,0.636565) (2.49,0.588837) (2.5,0.443282) (2.51,0.493382) (2.52,0.506797) (2.53,0.581354) (2.54,0.496941) (2.55,0.253558) (2.56,0.340343) (2.57,0.25593) (2.58,0.391629) (2.59,0.0748759) (2.6,0.17389) (2.61,0.334046) (2.62,0.18849) (2.63,0.214133) (2.64,0.28869) (2.65,0.240963) (2.66,-0.0391054) (2.67,0.0109945) (2.68,0.110008) (2.69,0.294621) (2.7,0.210208) (2.71,0.199166) (2.72,0.273723) (2.73,0.470564) (2.74,0.483979) (2.75,0.570764) (2.76,0.498579) (2.77,0.524222) (2.78,0.317525) (2.79,0.392082) (2.8,0.601152) (2.81,0.443368) (2.82,0.627981) (2.83,0.678081) (2.84,1.00943) (2.85,0.851651) (2.86,0.644954) (2.87,0.756196) (2.88,0.757382) (2.89,0.831939) (2.9,0.723069) (2.91,0.858768) (2.92,0.847726) (2.93,0.836684) (2.94,0.691129) (2.95,0.936884) (2.96,0.962527) (2.97,1.02486) (2.98,0.977127) (2.99,0.9294) (3,0.832758) (3.01,0.931772) (3.02,0.908501) (3.03,0.885231) (3.04,0.788589) (3.05,0.924288) (3.06,0.778733) (3.07,0.657635) (3.08,0.768877) (3.09,0.806749) (3.1,0.673422) (3.11,0.906948) (3.12,0.908135) (3.13,0.872635) (3.14,0.947192) (3.15,0.93615) (3.16,0.998478) (3.17,0.999665) (3.18,0.878566) (3.19,0.879752) (3.2,0.697512) (3.21,0.68647) (3.22,0.748798) (3.23,0.737756) (3.24,0.445459) (3.25,0.324361) (3.26,0.423375) (3.27,0.583531) (3.28,0.743687) (3.29,0.720416) (3.3,0.464805) (3.31,0.600504) (3.32,0.589461) (3.33,0.835217) (3.34,0.469549) (3.35,0.421822) (3.36,0.435236) (3.37,0.485336) (3.38,0.621035) (3.39,0.622221) (3.4,0.513352) (3.41,0.538995) (3.42,0.430125) (3.43,0.431311) (3.44,0.297984) (3.45,0.103515) (3.46,0.0680161) (3.47,0.00806006) (3.48,-0.198637) (3.49,-0.172994) (3.5,-0.0495237) (3.51,-0.0972513) (3.52,0.172961) (3.53,-0.00927985) (3.54,0.114191) (3.55,0.188748) (3.56,0.104335) (3.57,0.215577) (3.58,0.24122) (3.59,0.217949) (3.6,0.0968509) (3.61,0.0735801) (3.62,-0.133117) (3.63,0.100409) (3.64,0.0893671) (3.65,0.139467) (3.66,0.214024) (3.67,0.239667) (3.68,0.387594) (3.69,0.413237) (3.7,0.255454) (3.71,0.0854418) (3.72,0.221141) (3.73,0.136728) (3.74,0.0890003) (3.75,0.102415) (3.76,0.0669157) (3.77,-0.115325) (3.78,0.0203742) (3.79,0.0827026) (3.8,0.169488) (3.81,0.0973034) (3.82,0.135175) (3.83,0.111904) (3.84,0.186461) (3.85,0.0531343) (3.86,0.0910058) (3.87,-0.0178641) (3.88,0.325719) (3.89,0.0701074) (3.9,0.218035) (3.91,0.0602513) (3.92,-0.0241617) (3.93,-0.194174) (3.94,-0.192988) (3.95,-0.00837477) (3.96,-0.0561024) (3.97,-0.10383) (3.98,0.0930113) (3.99,-0.0403154) (4,-0.124728) (4.01,-0.184684) (4.02,-0.269097) (4.03,-0.108941) (4.04,-0.168898) (4.05,0.00348687) (4.06,0.0413584) (4.07,-0.0797399) (4.08,-0.127468) (4.09,-0.199652) (4.1,-0.0761814) (4.11,-0.221737) (4.12,-0.0493522) (4.13,0.06189) (4.14,0.136447) (4.15,0.0153485) (4.16,0.175504) (4.17,0.164462) (4.18,0.202334) (4.19,0.240205) (4.2,0.314762) (4.21,0.205892) (4.22,0.060337) (4.23,0.24495) (4.24,-0.0718036) (4.25,0.0883524) (4.26,0.101767) (4.27,0.298608) (4.28,0.458764) (4.29,0.129782) (4.3,0.0820547) (4.31,0.217754) (4.32,0.0477417) (4.33,0.0733848) (4.34,0.0867994) (4.35,0.124671) (4.36,-0.0575696) (4.37,0.249328) (4.38,0.2016) (4.39,0.459584) (4.4,0.411856) (4.41,0.278529) (4.42,0.536513) (4.43,0.562156) (4.44,0.379915) (4.45,0.564528) (4.46,0.541257) (4.47,0.322332) (4.48,0.372431) (4.49,0.226876) (4.5,0.460403) (4.51,0.498274) (4.52,0.548374) (4.53,0.329449) (4.54,0.575204) (4.55,0.564161) (4.56,0.602033) (4.57,0.652133) (4.58,0.628862) (4.59,0.361022) (4.6,0.521178) (4.61,0.497908) (4.62,0.315667) (4.63,0.475823) (4.64,0.562608) (4.65,0.637165) (4.66,0.711722) (4.67,0.272684) (4.68,0.127129) (4.69,0.226143) (4.7,0.264014) (4.71,0.19183) (4.72,0.290843) (4.73,0.561055) (4.74,0.537785) (4.75,0.428915) (4.76,0.35673) (4.77,0.370145) (4.78,0.38356) (4.79,0.409203) (4.8,0.581587) (4.81,0.533859) (4.82,0.486132) (4.83,0.536232) (4.84,0.43959) (4.85,0.245121) (4.86,0.38082) (4.87,0.259722) (4.88,0.309822) (4.89,0.262094) (4.9,0.495621) (4.91,0.39898) (4.92,0.351252) (4.93,0.291296) (4.94,0.182426) (4.95,0.342582) (4.96,0.197027) (4.97,0.22267) (4.98,0.248313) (4.99,0.225042) (5,0.336284) (5.01,0.2641) (5.02,0.460941) (5.03,0.43767) (5.04,0.475542) (5.05,0.415586) (5.06,0.294488) (5.07,0.454644) (5.08,0.333545) (5.09,0.310275) (5.1,0.274775) (5.11,0.324875) (5.12,-0.0285634) (5.13,0.0826788) (5.14,0.157236) (5.15,0.354077) (5.16,0.318578) (5.17,0.11188) (5.18,0.0519242) (5.19,-0.167002) (5.2,-0.153587) (5.21,-0.201315) (5.22,-0.102301) (5.23,0.155683) (5.24,0.14464) (5.25,0.0357705) (5.26,0.0369567) (5.27,-0.120827) (5.28,0.173842) (5.29,0.333998) (5.3,0.322956) (5.31,0.214086) (5.32,0.080759) (5.33,0.11863) (5.34,0.0586744) (5.35,0.157688) (5.36,0.403443) (5.37,0.539142) (5.38,0.405816) (5.39,0.492601) (5.4,0.457102) (5.41,0.666171) (5.42,0.593987) (5.43,0.558488) (5.44,0.315105) (5.45,0.145093) (5.46,0.21965) (5.47,0.220836) (5.48,0.332078) (5.49,0.162066) (5.5,0.199937) (5.51,0.286723) (5.52,0.165624) (5.53,0.362466) (5.54,0.131311) (5.55,0.108041) (5.56,0.292654) (5.57,-0.0363283) (5.58,-0.243026) (5.59,0.0761) (5.6,-0.0572267) (5.61,0.200757) (5.62,0.128572) (5.63,0.203129) (5.64,0.143173) (5.65,0.0709886) (5.66,0.0477179) (5.67,0.195645) (5.68,0.0378617) (5.69,0.112419) (5.7,0.0157772) (5.71,0.0414202) (5.72,0.299404) (5.73,-0.00512122) (5.74,-0.187362) (5.75,0.131764) (5.76,0.181864) (5.77,0.18305) (5.78,-0.170389) (5.79,0.160966) (5.8,-0.0335034) (5.81,0.0777388) (5.82,0.127839) (5.83,-0.127772) (5.84,-0.151043) (5.85,-0.235456) (5.86,-0.295412) (5.87,-0.220855) (5.88,-0.158527) (5.89,-0.316311) (5.9,-0.253982) (5.91,-0.118283) (5.92,0.00518744) (5.93,-0.103682) (5.94,0.0931589) (5.95,0.13103) (5.96,0.230044) (5.97,0.0967174) (5.98,-0.134437) (5.99,0.0990897) (6,-0.0464655) (6.01,0.0158629) (6.02,-0.105235) (6.03,-0.116278) (6.04,0.0438783) (6.05,0.204034) (6.06,0.20522) (6.07,0.402062) (6.08,0.354334) (6.09,0.526718) (6.1,0.540133) (6.11,0.467949) (6.12,0.652561) (6.13,0.360265) (6.14,0.483736) (6.15,0.228124) (6.16,0.400509) (6.17,0.548436) (6.18,0.0849415) (6.19,0.183955) (6.2,0.35634) (6.21,0.0640431) (6.22,0.273113) (6.23,0.0786438) (6.24,0.0920585) (6.25,0.203301) (6.26,0.18003) (6.27,0.0956169) (6.28,0.402514) (6.29,0.305873) (6.3,0.22146) (6.31,0.11259) (6.32,0.346117) (6.33,0.286161) (6.34,0.311804) (6.35,0.374132) (6.36,0.460917) (6.37,0.364276) (6.38,0.46329) (6.39,0.660131) (6.4,0.514576) (6.41,0.515762) (6.42,0.296836) (6.43,0.298022) (6.44,0.348122) (6.45,0.300395) (6.46,0.411637) (6.47,0.718534) (6.48,0.670806) (6.49,0.586393) (6.5,0.685407) (6.51,0.649908) (6.52,0.565495) (6.53,0.591138) (6.54,0.31107) (6.55,0.385627) (6.56,0.227843) (6.57,0.400228) (6.58,0.340272) (6.59,0.512656) (6.6,0.342644) (6.61,0.453886) (6.62,0.23496) (6.63,0.309517) (6.64,0.347389) (6.65,0.262976) (6.66,0.447588) (6.67,0.302033) (6.68,0.315448) (6.69,0.524518) (6.7,0.489018) (6.71,0.429062) (6.72,0.18568) (6.73,0.370292) (6.74,0.334793) (6.75,0.164781) (6.76,0.361623) (6.77,0.264981) (6.78,0.290624) (6.79,0.377409) (6.8,0.244083) (6.81,0.367553) (6.82,0.344283) (6.83,0.198727) (6.84,0.285513) (6.85,0.311156) (6.86,0.190057) (6.87,0.0445023) (6.88,0.0212316) (6.89,0.193616) (6.9,-0.0864521) (6.91,-0.183094) (6.92,-0.0229376) (6.93,-0.229635) (6.94,-0.21622) (6.95,-0.043836) (6.96,-0.164934) (6.97,-0.0659205) (6.98,-0.223704) (6.99,-0.00240596) (7,0.108836) (7.01,0.0977939) (7.02,0.196808) (7.03,0.039024) (7.04,0.113581) (7.05,0.102539) (7.06,0.0425825) (7.07,0.104911) (7.08,0.411808) (7.09,0.596421) (7.1,0.719892) (7.11,0.354224) (7.12,0.367639) (7.13,0.405511) (7.14,0.443382) (7.15,0.27337) (7.16,0.433526) (7.17,0.459169) (7.18,0.631553) (7.19,0.583826) (7.2,0.315986) (7.21,0.537284) (7.22,0.477328) (7.23,0.405144) (7.24,0.369645) (7.25,0.0162058) (7.26,0.408702) (7.27,0.385432) (7.28,0.44776) (7.29,0.400032) (7.3,0.462361) (7.31,0.744801) (7.32,0.660388) (7.33,0.722717) (7.34,0.674989) (7.35,0.688404) (7.36,0.457249) (7.37,0.397293) (7.38,0.594134) (7.39,0.729834) (7.4,0.584278) (7.41,0.450952) (7.42,0.51328) (7.43,0.575608) (7.44,0.662394) (7.45,0.675808) (7.46,0.640309) (7.47,0.494754) (7.48,0.703824) (7.49,0.460441) (7.5,0.498312) (7.51,0.585098) (7.52,0.598512) (7.53,0.795354) (7.54,0.600885) (7.55,0.565385) (7.56,0.701085) (7.57,0.592215) (7.58,0.752371) (7.59,0.88807) (7.6,1.13382) (7.61,0.963813) (7.62,0.854943) (7.63,1.05178) (7.64,0.942914) (7.65,1.09084) (7.66,0.859688) (7.67,0.665219) (7.68,0.764232) (7.69,0.789875) (7.7,0.656549) (7.71,0.841161) (7.72,0.548865) (7.73,0.733478) (7.74,0.746892) (7.75,0.833678) (7.76,0.798178) (7.77,0.640395) (7.78,0.519297) (7.79,0.557168) (7.8,0.362699) (7.81,0.632911) (7.82,0.634097) (7.83,0.867624) (7.84,0.660926) (7.85,0.625427) (7.86,0.663299) (7.87,0.737856) (7.88,0.861326) (7.89,0.984797) (7.9,0.985983) (7.91,0.913798) (7.92,0.914985) (7.93,1.014) (7.94,0.966271) (7.95,0.918543) (7.96,0.968643) (7.97,0.810859) (7.98,0.750903) (7.99,0.82546) (8,0.900017) (8.01,0.986802) (8.02,0.694506) (8.03,0.573408) (8.04,0.599051) (8.05,0.673607) (8.06,0.491367) (8.07,0.492553) (8.08,0.310313) (8.09,0.323727) (8.1,0.165944) (8.11,0.326099) (8.12,0.559626) (8.13,0.414071) (8.14,0.733197) (8.15,0.453129) (8.16,0.637741) (8.17,0.565557) (8.18,0.627885) (8.19,0.629071) (8.2,0.838141) (8.21,0.7415) (8.22,0.767143) (8.23,0.951756) (8.24,0.8918) (8.25,0.8563) (8.26,0.759659) (8.27,0.968729) (8.28,0.957687) (8.29,1.1423) (8.3,0.996744) (8.31,0.936788) (8.32,0.852375) (8.33,0.890247) (8.34,0.671321) (8.35,0.64805) (8.36,0.563637) (8.37,0.674879) (8.38,0.590466) (8.39,0.640566) (8.4,0.409412) (8.41,0.349456) (8.42,0.338414) (8.43,0.254001) (8.44,0.0717602) (8.45,0.170774) (8.46,0.147503) (8.47,0.0141765) (8.48,0.0153626) (8.49,0.151062) (8.5,0.0788772) (8.51,0.324632) (8.52,0.240219) (8.53,0.52266) (8.54,0.267049) (8.55,0.207092) (8.56,0.342792) (8.57,0.478491) (8.58,0.210651) (8.59,0.0650958) (8.6,0.164109) (8.61,0.458778) (8.62,0.374365) (8.63,0.204353) (8.64,0.43788) (8.65,0.622493) (8.66,0.611451) (8.67,0.392525) (8.68,0.381482) (8.69,0.174785) (8.7,-0.0319125) (8.71,0.0915582) (8.72,-0.0295401) (8.73,0.0205598) (8.74,-0.124995) (8.75,0.108531) (8.76,0.146403) (8.77,0.0986751) (8.78,0.0754044) (8.79,0.467901) (8.8,0.028863) (8.81,0.213476) (8.82,0.263576) (8.83,0.240305) (8.84,0.339319) (8.85,0.523932) (8.86,0.366148) (8.87,0.538532) (8.88,0.466348) (8.89,0.430849) (8.9,0.39535) (8.91,0.482135) (8.92,0.287666) (8.93,0.288852) (8.94,0.583521) (8.95,0.230082) (8.96,0.341324) (8.97,0.195769) (8.98,0.184727) (8.99,0.124771) (9,0.370526) (9.01,0.579596) (9.02,0.580782) (9.03,0.508597) (9.04,0.277443) (9.05,0.449827) (9.06,0.340957) (9.07,0.488885) (9.08,0.355558) (9.09,0.564628) (9.1,0.455758) (9.11,0.554772) (9.12,0.629329) (9.13,0.59383) (9.14,0.839585) (9.15,0.706258) (9.16,0.719672) (9.17,0.659716) (9.18,0.563075) (9.19,0.41752) (9.2,0.785559) (9.21,0.725603) (9.22,0.751246) (9.23,0.838032) (9.24,0.643563) (9.25,0.473551) (9.26,0.60925) (9.27,0.50038) (9.28,0.464881) (9.29,0.60058) (9.3,0.699593) (9.31,0.725236) (9.32,0.726423) (9.33,0.800979) (9.34,0.53314) (9.35,0.42427) (9.36,0.290943) (9.37,0.402185) (9.38,0.513428) (9.39,0.575756) (9.4,0.479115) (9.41,0.688184) (9.42,0.872797) (9.43,0.531587) (9.44,0.520545) (9.45,0.558416) (9.46,0.645201) (9.47,0.719758) (9.48,0.647574) (9.49,0.636531) (9.5,0.894515) (9.51,0.859016) (9.52,0.750146) (9.53,0.726875) (9.54,0.776975) (9.55,0.668105) (9.56,0.620378) (9.57,0.487051) (9.58,0.598293) (9.59,0.489423) (9.6,0.625122) (9.61,0.724136) (9.62,0.66418) (9.63,0.677595) (9.64,0.960035) (9.65,1.01013) (9.66,0.962407) (9.67,0.975822) (9.68,0.977008) (9.69,1.06379) (9.7,0.893781) (9.71,0.992795) (9.72,0.749412) (9.73,0.530486) (9.74,0.458302) (9.75,0.863027) (9.76,0.7297) (9.77,0.730886) (9.78,0.609788) (9.79,0.708802) (9.8,0.746673) (9.81,0.796773) (9.82,0.71236) (9.83,0.505663) (9.84,0.396793) (9.85,0.385751) (9.86,0.179053) (9.87,0.180239) (9.88,-0.0876003) (9.89,0.304896) (9.9,0.257169) (9.91,0.270583) (9.92,0.50411) (9.93,0.358555) (9.94,0.530939) (9.95,0.507668) (9.96,0.716738) (9.97,0.779066) (9.98,0.804709) (9.99,0.683611) (10,0.525828)};

\foreach \x in {0, 1, 2, 3, 4, 5, 6, 7, 8, 9, 10}
\draw[shift={(\x,0)}] (0,0) node[circle,fill,inner sep=0.5pt] {};

\foreach \y in {-1,0,1}
\draw[shift={(0,\y)}] (0,0) node[circle,fill,inner sep=0.5pt] {};

\end{tikzpicture}
\caption{Partial sums of $\mu^2(n)$ : depiction of \eqref{eq:broww} with $x = 2 \times 10^{15}$, $H = 44721359$ and $0 \leq t \leq 10$.} \label{Figure1}
\begin{tikzpicture}[yscale = 0.25]
\draw [blue, xshift=0cm] plot [smooth, tension=1] coordinates { (0.0,0.0)
(0.01,-0.0231961) (0.02,-0.172834) (0.03,0.267589) (0.04,0.202246) (0.05,0.0420715) (0.06,0.108438) (0.07,0.606814) (0.08,0.90499) (0.09,1.38756) (0.1,0.98504) (0.11,0.271683) (0.12,0.364392) (0.13,-0.701948) (0.14,0.391757) (0.15,0.00504126) (0.16,0.577175) (0.17,0.611931) (0.18,0.973328) (0.19,0.797348) (0.2,0.531806) (0.21,1.39897) (0.22,1.93422) (0.23,2.58011) (0.24,1.91417) (0.25,1.96474) (0.26,2.31033) (0.27,2.56109) (0.28,2.11115) (0.29,1.60853) (0.3,1.60114) (0.31,1.09852) (0.32,1.16489) (0.33,1.24179) (0.34,0.791853) (0.35,0.910904) (0.36,0.993076) (0.37,1.74433) (0.38,1.83704) (0.39,1.46613) (0.4,2.08041) (0.41,1.1089) (0.42,1.37547) (0.43,1.57355) (0.44,1.49767) (0.45,1.32169) (0.46,1.00346) (0.47,0.595672) (0.48,0.82009) (0.49,0.833773) (0.5,1.56396) (0.51,1.82526) (0.52,1.33317) (0.53,0.456495) (0.54,1.1498) (0.55,1.66398) (0.56,1.488) (0.57,1.53856) (0.58,1.12551) (0.59,0.765133) (0.6,0.668179) (0.61,0.750351) (0.62,0.437393) (0.63,1.18338) (0.64,0.949451) (0.65,0.694445) (0.66,0.745007) (0.67,1.40143) (0.68,1.70488) (0.69,1.5289) (0.7,2.3065) (0.71,2.59414) (0.72,1.88605) (0.73,2.06832) (0.74,2.04513) (0.75,1.55831) (0.76,2.55192) (0.77,1.7648) (0.78,2.04191) (0.79,2.5034) (0.8,2.48548) (0.81,2.09876) (0.82,1.63829) (0.83,1.63616) (0.84,2.30313) (0.85,3.1018) (0.86,2.7625) (0.87,2.46535) (0.88,2.79514) (0.89,2.361) (0.9,2.71187) (0.91,2.24086) (0.92,2.12283) (0.93,2.41574) (0.94,3.23022) (0.95,3.79708) (0.96,4.15848) (0.97,4.30914) (0.98,5.02352) (0.99,5.48502) (1,5.55665) (1.01,5.01715) (1.02,5.15728) (1.03,4.60197) (1.04,4.03086) (1.05,3.34385) (1.06,3.76846) (1.07,4.47758) (1.08,4.1172) (1.09,4.03605) (1.1,4.52916) (1.11,3.83688) (1.12,4.43009) (1.13,4.61763) (1.14,3.95695) (1.15,3.62292) (1.16,3.15718) (1.17,2.63875) (1.18,2.67351) (1.19,3.18769) (1.2,3.91788) (1.21,3.78404) (1.22,4.27188) (1.23,4.50684) (1.24,4.64696) (1.25,4.92407) (1.26,6.10206) (1.27,5.86813) (1.28,5.88182) (1.29,6.0957) (1.3,6.70471) (1.31,6.14941) (1.32,5.97869) (1.33,6.31375) (1.34,6.13777) (1.35,6.13565) (1.36,6.28631) (1.37,5.97862) (1.38,4.87013) (1.39,5.16831) (1.4,5.10823) (1.41,5.51705) (1.42,5.34633) (1.43,6.14501) (1.44,6.20611) (1.45,5.28201) (1.46,4.91637) (1.47,4.77727) (1.48,4.95954) (1.49,4.63605) (1.5,4.42846) (1.51,4.29989) (1.52,4.73505) (1.53,4.45897) (1.54,3.90893) (1.55,4.71288) (1.56,4.86881) (1.57,4.80873) (1.58,4.1744) (1.59,4.98361) (1.6,4.93934) (1.61,4.56843) (1.62,4.28708) (1.63,3.82661) (1.64,3.82449) (1.65,3.54841) (1.66,3.56209) (1.67,3.26494) (1.68,3.76331) (1.69,3.70324) (1.7,3.68004) (1.71,3.43031) (1.72,3.84439) (1.73,3.73163) (1.74,3.75585) (1.75,3.975) (1.76,3.57775) (1.77,2.95922) (1.78,2.62519) (1.79,2.69682) (1.8,2.779) (1.81,3.53026) (1.82,3.88638) (1.83,3.16249) (1.84,2.29635) (1.85,2.58399) (1.86,2.44489) (1.87,2.77994) (1.88,2.54601) (1.89,2.09608) (1.9,1.39852) (1.91,1.60714) (1.92,1.67351) (1.93,1.47118) (1.94,1.3953) (1.95,1.08235) (1.96,0.606067) (1.97,0.877901) (1.98,0.754606) (1.99,0.0939329) (2,0.223521) (2.01,0.843069) (2.02,0.941047) (2.03,1.34459) (2.04,1.36881) (2.05,1.48786) (2.06,1.08007) (2.07,0.703893) (2.08,0.132784) (2.09,0.488912) (2.1,0.607963) (2.11,0.0157798) (2.12,0.972506) (2.13,1.18112) (2.14,0.146391) (2.15,0.165342) (2.16,0.331809) (2.17,-0.218228) (2.18,-0.678701) (2.19,-1.04961) (2.2,-0.556504) (2.21,-1.38577) (2.22,-1.34574) (2.23,-1.32679) (2.24,-1.53438) (2.25,-1.68929) (2.26,-0.553433) (2.27,-0.529214) (2.28,-0.673583) (2.29,-1.08137) (2.3,-1.28369) (2.31,-1.06981) (2.32,-0.186843) (2.33,0.448511) (2.34,0.193506) (2.35,-0.767466) (2.36,-0.23748) (2.37,-0.676879) (2.38,-0.905543) (2.39,-1.31333) (2.4,-2.07937) (2.41,-2.43975) (2.42,-2.36284) (2.43,-2.80224) (2.44,-2.57782) (2.45,-3.19635) (2.46,-3.93078) (2.47,-3.49562) (2.48,-3.08681) (2.49,-2.28814) (2.5,-2.62744) (2.51,-1.92886) (2.52,-2.36299) (2.53,-1.99633) (2.54,-1.66654) (2.55,-2.21131) (2.56,-2.14494) (2.57,-2.25243) (2.58,-2.41788) (2.59,-1.55598) (2.6,-0.941699) (2.61,-1.24412) (2.62,-1.49386) (2.63,-1.63296) (2.64,-1.84582) (2.65,-1.18939) (2.66,-0.859602) (2.67,-1.11988) (2.68,-2.03343) (2.69,-2.18307) (2.7,-2.08509) (2.71,-3.28841) (2.72,-3.20097) (2.73,-3.25051) (2.74,-2.46764) (2.75,-2.53825) (2.76,-2.85648) (2.77,-3.23792) (2.78,-4.23051) (2.79,-5.07557) (2.8,-5.45702) (2.81,-5.6488) (2.82,-6.09874) (2.83,-6.05872) (2.84,-6.01342) (2.85,-6.20521) (2.86,-6.37065) (2.87,-5.85647) (2.88,-6.67519) (2.89,-6.67205) (2.9,-6.66363) (2.91,-7.22421) (2.92,-6.84174) (2.93,-6.91235) (2.94,-7.21477) (2.95,-7.68578) (2.96,-7.36653) (2.97,-7.56358) (2.98,-7.63419) (2.99,-7.97349) (3,-8.07045) (3.01,-7.73012) (3.02,-7.25282) (3.03,-6.97572) (3.04,-6.99891) (3.05,-6.62698) (3.06,-7.06111) (3.07,-7.58481) (3.08,-7.62908) (3.09,-6.8146) (3.1,-6.64286) (3.11,-7.04538) (3.12,-7.31619) (3.13,-7.371) (3.14,-6.98326) (3.15,-7.00119) (3.16,-7.272) (3.17,-6.7947) (3.18,-6.40169) (3.19,-6.09297) (3.2,-5.9107) (3.21,-5.29115) (3.22,-5.80958) (3.23,-6.05932) (3.24,-5.85597) (3.25,-5.34706) (3.26,-4.78546) (3.27,-5.20906) (3.28,-5.74329) (3.29,-4.87086) (3.3,-4.60429) (3.31,-4.79081) (3.32,-4.44521) (3.33,-4.853) (3.34,-4.85513) (3.35,-4.83091) (3.36,-4.70659) (3.37,-4.8931) (3.38,-4.23668) (3.39,-3.59605) (3.4,-4.1777) (3.41,-3.40537) (3.42,-3.50232) (3.43,-3.44122) (3.44,-2.93758) (3.45,-2.82379) (3.46,-2.58357) (3.47,-1.58997) (3.48,-1.77648) (3.49,-2.12632) (3.5,-2.1021) (3.51,-2.10422) (3.52,-1.56897) (3.53,-1.56582) (3.54,-1.32033) (3.55,-1.57533) (3.56,-2.13064) (3.57,-1.96417) (3.58,-1.3657) (3.59,-1.47846) (3.6,-0.653438) (3.61,-0.155063) (3.62,0.0482826) (3.63,0.657295) (3.64,0.249506) (3.65,0.721539) (3.66,0.398044) (3.67,0.506559) (3.68,1.04708) (3.69,0.565535) (3.7,0.674049) (3.71,0.66139) (3.72,1.08601) (3.73,1.67921) (3.74,1.59807) (3.75,1.6539) (3.76,1.09332) (3.77,1.28613) (3.78,1.12069) (3.79,1.40306) (3.8,1.83822) (3.81,2.08898) (3.82,2.10793) (3.83,1.04686) (3.84,0.628531) (3.85,0.842413) (3.86,1.25123) (3.87,1.89712) (3.88,1.67372) (3.89,2.46713) (3.9,1.99612) (3.91,2.44181) (3.92,2.04456) (3.93,2.14253) (3.94,2.093) (3.95,1.8696) (3.96,1.81479) (3.97,1.84428) (3.98,1.34693) (3.99,1.22364) (4,0.984435) (4.01,1.14563) (4.02,0.859017) (4.03,0.567133) (4.04,-0.272666) (4.05,0.246784) (4.06,0.581839) (4.07,1.08548) (4.08,0.577595) (4.09,0.422689) (4.1,0.504861) (4.11,-0.0557122) (4.12,-0.452965) (4.13,-0.97139) (4.14,-1.60572) (4.15,-1.23906) (4.16,-1.77329) (4.17,-1.90712) (4.18,-1.5826) (4.19,-1.36345) (4.2,-1.03893) (4.21,-0.566897) (4.22,-0.0632526) (4.23,-0.860904) (4.24,-0.394138) (4.25,-0.527971) (4.26,0.170605) (4.27,0.231703) (4.28,0.946083) (4.29,1.4708) (4.3,1.38965) (4.31,0.786933) (4.32,1.08511) (4.33,0.893324) (4.34,0.675198) (4.35,-0.417484) (4.36,0.0598185) (4.37,-0.785248) (4.38,-0.945422) (4.39,-0.868519) (4.4,-0.770541) (4.41,-1.33638) (4.42,-1.4386) (4.43,-1.33009) (4.44,-2.35428) (4.45,-2.35641) (4.46,-2.31638) (4.47,-2.36592) (4.48,-2.07828) (4.49,-2.53348) (4.5,-2.2037) (4.51,-2.33226) (4.52,-2.53458) (4.53,-2.30489) (4.54,-1.30075) (4.55,-0.918281) (4.56,-0.925672) (4.57,-1.34927) (4.58,-1.43041) (4.59,-0.162852) (4.6,-0.623325) (4.61,-0.830915) (4.62,-1.05431) (4.63,-0.287246) (4.64,-0.800403) (4.65,-0.365248) (4.66,-0.478007) (4.67,-0.464324) (4.68,-0.245174) (4.69,-0.0260231) (4.7,0.235275) (4.71,0.238421) (4.72,0.152004) (4.73,0.392228) (4.74,-0.263176) (4.75,0.108758) (4.76,0.528108) (4.77,0.194076) (4.78,-0.329618) (4.79,-0.610965) (4.8,0.103415) (4.81,-0.135785) (4.82,-0.222202) (4.83,-0.392913) (4.84,0.105463) (4.85,0.271929) (4.86,0.285612) (4.87,0.562715) (4.88,0.750255) (4.89,0.458371) (4.9,-0.0811283) (4.91,1.04419) (4.92,0.815524) (4.93,1.19273) (4.94,1.36973) (4.95,1.01989) (4.96,0.95455) (4.97,0.831254) (4.98,0.697422) (4.99,0.547784) (5,1.30958) (5.01,1.24424) (5.02,1.1947) (5.03,1.47707) (5.04,1.97018) (5.05,2.65822) (5.06,2.33999) (5.07,2.58548) (5.08,2.71507) (5.09,3.22398) (5.1,2.59492) (5.11,2.19767) (5.12,2.40628) (5.13,2.48319) (5.14,2.47053) (5.15,2.45787) (5.16,2.1133) (5.17,2.17967) (5.18,2.16701) (5.19,2.73387) (5.2,3.02678) (5.21,2.79285) (5.22,2.72224) (5.23,3.0731) (5.24,2.82863) (5.25,2.33655) (5.26,3.06146) (5.27,3.22266) (5.28,3.62094) (5.29,3.8717) (5.3,4.14353) (5.31,4.24151) (5.32,4.30788) (5.33,4.57444) (5.34,4.44061) (5.35,3.86423) (5.36,3.66718) (5.37,3.23305) (5.38,3.78938) (5.39,4.08755) (5.4,3.58493) (5.41,3.15607) (5.42,3.35942) (5.43,2.63552) (5.44,1.98539) (5.45,2.13078) (5.46,2.48164) (5.47,1.97375) (5.48,2.31407) (5.49,1.44267) (5.5,1.98319) (5.51,1.52798) (5.52,1.61016) (5.53,0.744017) (5.54,0.841994) (5.55,-0.250687) (5.56,-0.521498) (5.57,-0.560499) (5.58,-0.573158) (5.59,-0.106393) (5.6,0.423594) (5.61,-0.442546) (5.62,-0.797652) (5.63,-1.06846) (5.64,-0.638576) (5.65,-0.846166) (5.66,0.568912) (5.67,0.456153) (5.68,0.949261) (5.69,0.94187) (5.7,0.776427) (5.71,1.13782) (5.72,1.00926) (5.73,1.00187) (5.74,1.40015) (5.75,1.17148) (5.76,1.93855) (5.77,1.32002) (5.78,1.46015) (5.79,2.34838) (5.8,2.8573) (5.81,2.92366) (5.82,3.00584) (5.83,3.36723) (5.84,3.30716) (5.85,3.5105) (5.86,3.41355) (5.87,2.76342) (5.88,3.35662) (5.89,3.73382) (5.9,4.3165) (5.91,4.56199) (5.92,4.91812) (5.93,4.99502) (5.94,4.52928) (5.95,4.78004) (5.96,4.65148) (5.97,4.44389) (5.98,4.40489) (5.99,4.9296) (6,4.20571) (6.01,3.62933) (6.02,3.45335) (6.03,4.08344) (6.04,3.52287) (6.05,3.22571) (6.06,2.25947) (6.07,2.34691) (6.08,2.35533) (6.09,2.12666) (6.1,2.00337) (6.11,2.1593) (6.12,2.15191) (6.13,2.19193) (6.14,1.69985) (6.15,1.59763) (6.16,1.7957) (6.17,2.00432) (6.18,1.79146) (6.19,1.73139) (6.2,1.77141) (6.21,2.25398) (6.22,1.47214) (6.23,0.716632) (6.24,0.672363) (6.25,0.923124) (6.26,0.0464467) (6.27,0.0127141) (6.28,-0.0104818) (6.29,-0.565786) (6.3,-0.125363) (6.31,-0.411979) (6.32,0.0758605) (6.33,-0.584812) (6.34,-0.23922) (6.35,-1.17385) (6.36,-1.70281) (6.37,-2.16855) (6.38,-1.9336) (6.39,-1.8883) (6.4,-1.83247) (6.41,-1.34463) (6.42,-1.52588) (6.43,-2.0127) (6.44,-1.86204) (6.45,-2.32778) (6.46,-2.07702) (6.47,-2.18451) (6.48,-2.08653) (6.49,-2.34153) (6.5,-2.69664) (6.51,-2.69349) (6.52,-3.32256) (6.53,-3.10867) (6.54,-3.01596) (6.55,-3.10765) (6.56,-2.82528) (6.57,-3.57551) (6.58,-3.94642) (6.59,-3.87479) (6.6,-3.4449) (6.61,-3.78947) (6.62,-3.44915) (6.63,-3.27214) (6.64,-3.62198) (6.65,-4.09826) (6.66,-4.50078) (6.67,-4.32378) (6.68,-4.26795) (6.69,-4.01192) (6.7,-3.93501) (6.71,-3.87391) (6.72,-4.01301) (6.73,-3.50937) (6.74,-3.48515) (6.75,-2.92355) (6.76,-3.45252) (6.77,-3.78128) (6.78,-2.97734) (6.79,-3.27449) (6.8,-3.03953) (6.81,-2.84672) (6.82,-2.84885) (6.83,-2.76667) (6.84,-2.99007) (6.85,-3.01326) (6.86,-3.06807) (6.87,-2.69614) (6.88,-3.42003) (6.89,-3.03756) (6.9,-3.16612) (6.91,-2.68882) (6.92,-2.76997) (6.93,-3.02497) (6.94,-3.32739) (6.95,-3.27683) (6.96,-3.04715) (6.97,-3.20732) (6.98,-3.62565) (6.99,-3.4223) (7,-3.46657) (7.01,-3.34225) (7.02,-2.91236) (7.03,-2.57731) (7.04,-2.72168) (7.05,-2.53414) (7.06,-2.68904) (7.07,-3.44981) (7.08,-3.041) (7.09,-3.12215) (7.1,-2.99256) (7.11,-4.16954) (7.12,-4.74065) (7.13,-4.90609) (7.14,-4.9609) (7.15,-3.97256) (7.16,-3.85351) (7.17,-3.7766) (7.18,-3.81561) (7.19,-2.4058) (7.2,-2.57124) (7.21,-2.57863) (7.22,-3.11286) (7.23,-2.70405) (7.24,-2.35845) (7.25,-2.38692) (7.26,-2.08874) (7.27,-2.12247) (7.28,-1.41863) (7.29,-1.65256) (7.3,-2.33431) (7.31,-1.8412) (7.32,-2.41231) (7.33,-1.87179) (7.34,-2.00562) (7.35,-2.23955) (7.36,-1.97825) (7.37,-2.67054) (7.38,-2.13528) (7.39,-1.94774) (7.4,-2.74012) (7.41,-2.24175) (7.42,-0.895159) (7.43,-1.06587) (7.44,-1.30507) (7.45,-1.28612) (7.46,-1.67283) (7.47,-1.49056) (7.48,-1.66654) (7.49,-1.48427) (7.5,-1.62337) (7.51,-1.99955) (7.52,-1.02175) (7.53,0.0561508) (7.54,0.154128) (7.55,0.778946) (7.56,0.971754) (7.57,1.34896) (7.58,0.96751) (7.59,0.412206) (7.6,0.578672) (7.61,0.0286365) (7.62,0.395302) (7.63,1.273) (7.64,1.49742) (7.65,0.420546) (7.66,0.571208) (7.67,0.832506) (7.68,1.01478) (7.69,1.52369) (7.7,1.86928) (7.71,1.63535) (7.72,1.62796) (7.73,0.566889) (7.74,0.317152) (7.75,0.0568784) (7.76,-0.134906) (7.77,0.0947813) (7.78,0.0399751) (7.79,0.106342) (7.8,-0.433157) (7.81,1.10836) (7.82,0.737452) (7.83,-0.249861) (7.84,-0.0201736) (7.85,-0.0486378) (7.86,0.0651449) (7.87,0.521374) (7.88,0.772135) (7.89,0.833233) (7.9,0.678328) (7.91,0.597179) (7.92,0.510763) (7.93,0.287368) (7.94,0.559202) (7.95,0.298929) (7.96,0.354759) (7.97,1.50642) (7.98,0.945844) (7.99,1.20714) (8,1.04697) (8.01,0.776157) (8.02,0.684472) (8.03,0.434735) (8.04,0.353587) (8.05,0.388343) (8.06,0.955209) (8.07,1.29026) (8.08,1.8624) (8.09,1.76018) (8.1,0.920378) (8.11,0.359805) (8.12,0.215436) (8.13,0.68747) (8.14,0.411391) (8.15,0.419806) (8.16,0.870766) (8.17,0.368146) (8.18,0.0235776) (8.19,-0.0575706) (8.2,-0.507507) (8.21,-0.30943) (8.22,-0.295747) (8.23,-0.181964) (8.24,0.643053) (8.25,0.825325) (8.26,0.875887) (8.27,0.39434) (8.28,0.897984) (8.29,0.600832) (8.3,0.640857) (8.31,0.480682) (8.32,1.27409) (8.33,1.44056) (8.34,0.521732) (8.35,0.366826) (8.36,0.538561) (8.37,0.383655) (8.38,1.24028) (8.39,1.71232) (8.4,2.3424) (8.41,1.67646) (8.42,1.50048) (8.43,1.69329) (8.44,2.55519) (8.45,1.52573) (8.46,1.67639) (8.47,2.02198) (8.48,2.67314) (8.49,2.70263) (8.5,1.89971) (8.51,2.21369) (8.52,1.9903) (8.53,1.54036) (8.54,0.784856) (8.55,1.35699) (8.56,1.51292) (8.57,1.40016) (8.58,2.34108) (8.59,2.57604) (8.6,2.83734) (8.61,3.6044) (8.62,3.28617) (8.63,2.37789) (8.64,2.43899) (8.65,2.99531) (8.66,3.09329) (8.67,2.62228) (8.68,2.80982) (8.69,2.73921) (8.7,3.43252) (8.71,3.87294) (8.72,3.41774) (8.73,3.384) (8.74,3.9456) (8.75,3.82757) (8.76,4.02038) (8.77,3.48088) (8.78,3.59467) (8.79,3.24483) (8.8,3.78535) (8.81,4.02558) (8.82,4.19204) (8.83,3.67889) (8.84,4.53551) (8.85,5.03916) (8.86,4.29946) (8.87,4.20778) (8.88,4.27941) (8.89,4.41954) (8.9,4.30678) (8.91,4.62076) (8.92,4.35522) (8.93,4.05806) (8.94,4.27195) (8.95,4.20134) (8.96,4.04643) (8.97,3.58069) (8.98,3.25193) (8.99,3.19185) (9,2.86309) (9.01,2.81882) (9.02,3.5332) (9.03,4.05792) (9.04,4.20858) (9.05,4.34343) (9.06,4.83654) (9.07,4.65529) (9.08,5.41182) (9.09,5.29906) (9.1,5.25479) (9.11,6.03766) (9.12,6.63614) (9.13,6.1862) (9.14,6.83736) (9.15,6.17669) (9.16,6.7699) (9.17,6.13557) (9.18,6.59706) (9.19,6.85836) (9.2,6.76141) (9.21,6.73295) (9.22,7.60011) (9.23,7.91936) (9.24,8.07002) (9.25,7.25657) (9.26,7.28079) (9.27,7.23652) (9.28,7.1659) (9.29,6.57899) (9.3,6.42935) (9.31,6.42723) (9.32,5.94568) (9.33,5.7697) (9.34,5.57265) (9.35,4.84349) (9.36,5.13113) (9.37,4.90773) (9.38,4.63692) (9.39,4.18172) (9.4,3.94779) (9.41,3.84557) (9.42,4.2017) (9.43,4.34709) (9.44,3.40719) (9.45,3.48409) (9.46,4.21428) (9.47,4.08045) (9.48,3.92554) (9.49,5.16149) (9.5,5.10142) (9.51,4.57773) (9.52,4.83376) (9.53,4.75261) (9.54,5.47753) (9.55,4.01079) (9.56,3.77159) (9.57,3.69044) (9.58,3.37748) (9.59,3.06452) (9.6,2.69888) (9.61,2.59666) (9.62,2.76313) (9.63,2.66617) (9.64,1.89486) (9.65,1.27107) (9.66,0.910698) (9.67,0.629351) (9.68,-0.0365899) (9.69,0.145682) (9.7,0.0645339) (9.71,0.447005) (9.72,0.186731) (9.73,0.484908) (9.74,0.530202) (9.75,-0.0198342) (9.76,-0.116788) (9.77,0.808329) (9.78,0.885233) (9.79,0.535396) (9.8,0.412101) (9.81,0.120217) (9.82,0.481615) (9.83,0.990527) (9.84,0.572202) (9.85,0.427833) (9.86,0.768157) (9.87,0.41832) (9.88,0.526834) (9.89,0.656422) (9.9,0.533127) (9.91,0.64691) (9.92,0.492004) (9.93,-0.806145) (9.94,-1.26662) (9.95,-0.805121) (9.96,-0.696606) (9.97,-0.192962) (9.98,0.115752) (9.99,0.403392) (10,1.55505)};

\foreach \x in {0, 1, 2, 3, 4, 5, 6, 7, 8, 9, 10}
\draw[shift={(\x,0)}] (0,0) node[circle,fill,inner sep=0.5pt] {};

\foreach \y in {-8,-7,-6,-5,-4,-3,-2,-1,0,1,2,3,4,5,6,7,8}
\draw[shift={(0,\y)}] (0,0) node[circle,fill,inner sep=0.5pt] {};
\end{tikzpicture}
\caption{Partial sums of $\log p$ : depiction of \eqref{eq:browwprimes} with $x = 2 \times 10^{15}$, $H = 44721359$ and $0 \leq t \leq 10$.}\label{Figure2}
\end{figure}

\subsection{Fractional Brownian motion}
One notable feature of Theorems \ref{thm:main} and \ref{thm:main2} is that while the expected count of squarefrees in a short interval (or likewise arithmetic progression) is of order $H$, the variance of these counts is of order $H^{1/2}$. For many other natural arithmetic sequences (e.g. primes) one conjectures that the variance of counts is of the same order of magnitude as the expected value of counts.

That the variance is of order $H^{1/2}$ in Theorems \ref{thm:main} and \ref{thm:main2} speaks to the idea that the squarefree numbers are ``less random" than (for example) the primes (cf. \cite{cellarosi2013ergodic}). 
One may conjecture that higher moments are gaussian (see \cite{avdeeva2017ergodic} for numerical evidence). For $x$ drawn uniformly at random from $[X,2X]$, one may even make the stronger conjecture that the process
\begin{equation} \label{eq:broww}
t \mapsto \frac{1}{H^{1/4}} \sum_{x < n \leq x+tH} (\mu^2(n) - 1/\zeta(2))
\end{equation}
tends weakly, when suitably normalized by $H^{1/4}$, to a fractional Brownian motion with Hurst parameter $1/4$. See Figure \ref{Figure1} for an illustration of the evolution of the partial sums \eqref{eq:broww}. 
A formulation of this perspective seems to have been first made in \cite{grimmett1991asymptotics}. This is in contrast to the analogous process generated by prime-counting, where one may conjecture the appearance of Hurst parameter $1/2$ -- that is, usual Brownian motion. (See \cite{shevchenko2014fractional} for a survey on fractional Brownian motion.) The evolution of the process
\begin{equation} \label{eq:browwprimes}
t \mapsto \frac{1}{H^{1/2}} \sum_{x < p \leq x + t H} (\log p - 1).
\end{equation}
is depicted in Figure \ref{Figure2}. Both Figure \ref{Figure1} and Figure \ref{Figure2} depict the same range of parameters to make the comparison easier. The dots on Figure \ref{Figure1} and Figure \ref{Figure2} correspond to lattice points on the positive $x$-axis and on the (positive and negative) $y$-axis and indicate the difference in scales.

\subsection{Acknowledgments}
We would like to thank Bingrong Huang and Francesco Cellarosi for useful conversations, and the anonymous referees for their helpful comments. OG was supported by the European Research Council (ERC) under the European Union’s 2020 research and innovation programme (ERC grant agreement  n$^{\text{o}}$ 786758). KM was supported by Academy of Finland grant no. 285894. MR acknowledges partial support of a Sloan fellowship and of NSF Grant DMS-1902063. BR received partial support from NSF grant DMS-1854398 and an NSERC grant. Parts of this research were done during visits to Centre de Recherches Math\'{e}matiques and Oberwolfach and we thank these institutions for their hospitality.

\subsection{Conventions and Notations}
Throughout the rest of the paper we will allow the implicit constants in $\ll$ and $O(\cdot)$ to depend on $\varepsilon$. Furthermore the notation $n \sim N$ in the subscript of a sum will mean that $N \leq n < 2N$. 

\section{Proofs of Theorems \ref{thm:main} and \ref{thm:main2}}

We will show in this section how Theorems \ref{thm:main} and \ref{thm:main2} follow
from a number of technical propositions that are proven in Sections \ref{se:prop1}--\ref{se:prop2q}. 

The proof of Theorem \ref{thm:main} splits into two steps and depends on the identity
$$
\mu^2(m) = \sum_{nd^2 = m} \mu(d)
$$
and the following two propositions. 

\begin{proposition} \label{pr:prop1}
  Let $\varepsilon \in (0, \tfrac{1}{100})$ be given.
  Let $X \geq 1$ and $X^{\varepsilon} \leq H \leq X^{2/3 - \varepsilon}$. Let $H^{1+\varepsilon} \leq z \leq \min\{X/H^{1/2+\varepsilon}, H^{1/2-\varepsilon}X^{1/2}\}$. Then, as $X \rightarrow \infty$, 
  \begin{equation} \label{eq:prop1}  \frac{1}{X} \int_{X}^{2X} \Big | \sum_{\substack{x < n d^2 \leq x + H \\ d^2 \le z}} \mu(d) - H \sum_{d^2 \le z} \frac{\mu(d)}{d^2} \Big |^2 \, dx = C \sqrt{H} + O(H^{1/2-\varepsilon/10})
    \end{equation}
    with $C$ as in~\eqref{eq:C_dfn}.
\end{proposition}

\begin{proposition} \label{pr:prop2}
  Let $\varepsilon \in (0, \tfrac{1}{100})$ be given. Let $X \geq 1$ and $X^{\varepsilon} \leq H \leq X^{4/7 - \varepsilon}$. Let $z \geq H^{4/3+\varepsilon}$. Then
  \begin{equation} \label{eq:prop2}
  \frac{1}{X} \int_{X}^{2X} \Big | \sum_{\substack{x < n d^2 \leq x + H \\ d^2 > z}} \mu(d) - H \sum_{\substack{2X \geq d^2 > z}} \frac{\mu(d)}{d^2} \Big |^2 dx \ll H^{1/2 - \varepsilon / 8}. 
  \end{equation}
  Assuming the Lindel\"of Hypothesis, the claim holds in the wider range $X^\varepsilon \leq H \leq X^{2/3-\varepsilon}$ and $z \geq H^{1+\varepsilon}$.
\end{proposition}

Under the assumption of the Lindel\"of Hypothesis, the above propositions cover all the possible values of $d^2$ for $X^{\varepsilon} \leq H \leq X^{2/3-\varepsilon}$. However, unconditionally they cover all the possible values of $d^2$ only for $X^{\varepsilon} \leq H \leq X^{6/(11+12\varepsilon)}$.
It would be possible to improve on the exponent $4/7$ in Proposition~\ref{pr:prop2}, but this would not help. Similarly it should be possible to prove Proposition \ref{pr:prop1} only with the condition $H^{1 + \varepsilon} \leq z \leq X / H^{1/2 + \varepsilon}$ by adapting the proof of Proposition \ref{prop:prop1q} below. 

We note that only the terms $d$ with $d^2 \in [H^{1 - \varepsilon}, H^{1 + \varepsilon}]$ contribute to the main term $C \sqrt{H}$ in Proposition \ref{pr:prop1}. 

Roughly speaking Proposition \ref{pr:prop1} depends only on ``convex" inputs such as a Fourier expansion and a point-counting lemma, whereas Proposition \ref{pr:prop2} exploits large value estimates of Huxley and subconvexity and fourth moment estimates for the Riemann zeta-function. 

\begin{proof}[Proof of Theorem \ref{thm:main} assuming Proposition \ref{pr:prop1} and Proposition \ref{pr:prop2}]
Let $\varepsilon \in (0, \tfrac{1}{100})$. If $H \leq X^{\varepsilon}$ then the result already follows from Hall's theorem. We can therefore assume that $H > X^{\varepsilon}$. 

For $H \in [X^\varepsilon, X^{6/11-\varepsilon}]$, take $z = \min\{X/H^{1/2+\varepsilon}, H^{1/2-\varepsilon}X^{1/2}\}$. Note that $z \geq H^{4/3+\varepsilon}$. Denoting by $\mathcal{I}_1$ the left-hand side of \eqref{eq:prop1} and by $\mathcal{I}_2$ the left-hand side of \eqref{eq:prop2}, we get, using Cauchy-Schwarz, that
$$
\frac{1}{X} \int_{X}^{2X} \Big | \sum_{\substack{x < n d^2 \leq x + H}} \mu(d) - H \sum_{d^2 \leq 2X} \frac{\mu(d)}{d^2} \Big |^2 dx = \mathcal{I}_1 + O(\sqrt{\mathcal{I}_1 \mathcal{I}_2} + \mathcal{I}_2).
$$
Using the bounds in \eqref{eq:prop1} and \eqref{eq:prop2}, we conclude that
\begin{equation}
\label{eq:truncclaim}
\frac{1}{X} \int_{X}^{2X} \Big | \sum_{\substack{x < n d^2 \leq x + H}} \mu(d) - H \sum_{d^2 \leq 2X} \frac{\mu(d)}{d^2} \Big |^2 dx = C \sqrt{H} + O(H^{1/2 - \varepsilon / 16}).
\end{equation}
Notice that the tail $H\sum_{d^2 > 2X} \mu(d)/d^2$ is $\ll H/\sqrt{X}$. Hence the claim reduces to showing that
\[
\frac{1}{X} \int_{X}^{2X} \Big | \sum_{\substack{x < n d^2 \leq x + H}} \mu(d) - H \sum_{d^2 \leq 2X} \frac{\mu(d)}{d^2} \Big | \cdot \frac{H}{\sqrt{X}} + \Bigl(\frac{H}{\sqrt{X}}\Bigr)^2 dx \ll H^{1/2 - \varepsilon / 16}.
\]
Applying Cauchy-Schwarz and~\eqref{eq:truncclaim} we see that the left hand side
 is $\ll H^{1/4}(H/\sqrt{X}) + H^2/X \ll H^{1/2-\varepsilon/16}$ since $H \leq X^{2/3 - \varepsilon}$. 
\end{proof}

Likewise the proof of Theorem \ref{thm:main2} splits into two steps and depends on the following propositions.

\begin{proposition} \label{prop:prop1q}
Let $\varepsilon \in (0, \tfrac{1}{100})$.  Let $q$ be prime with $x^{1/3 + 30 \varepsilon} \leq q \leq x^{1 - \varepsilon}$ and let $(x/q)^{1 + \varepsilon} \leq z \leq x^{-\varepsilon} \cdot \sqrt{qx}$. Then 
    \begin{equation}
      \frac{1}{\varphi(q)}\sum_{(a,q) = 1} \Big | \sum_{\substack{d^2 n \leq x , \ d^2 < z \\ d^2 n \,\equiv\, a \Mod{q}}} \mu(d) - \frac{1}{\varphi(q)} \sum_{\substack{d^2 n \leq x, \ d^2 < z \\ (d^2 n,q)=1}} \mu(d) \Big |^2 = C \sqrt{x /q} + O\left((x/q)^{1/2-\varepsilon/16}\right)
  \end{equation}
      with $C$ as in~\eqref{eq:C_dfn}.
    \end{proposition}

\begin{proposition} \label{prop:prop2q}
Let $\varepsilon \in (0, \tfrac{1}{100})$. Let $x \geq 1$ and $x^{3/7 + \varepsilon} \leq q \leq x^{1 - \varepsilon}$. Let $z \geq (x/q)^{4/3+\varepsilon}$. Then
  \begin{equation}
  \label{eq:prop2qclaim}
  \frac{1}{\varphi(q)}\sum_{(a,q) = 1} \Big | \sum_{\substack{d^2 n \leq x, \ d^2 \geq z \\ d^2 n \,\equiv\, a \Mod{q}}} \mu(d) - \frac{1}{\varphi(q)} \sum_{\substack{d^2 n \leq x, \ d^2 \geq z \\ (d^2 n, q) = 1}} \mu(d) \Big |^2 \ll (x / q)^{1/2 - \varepsilon / 8}. 
  \end{equation}
  Assuming the Generalized Lindel\"of Hypothesis, the claim holds in the wider range $x^{1/3+\varepsilon} \leq q \leq x^{1-\varepsilon}$ and $z \geq (x/q)^{1+\varepsilon}$.
\end{proposition}

The proof of Proposition \ref{prop:prop1q} depends once again only on ``convex'' inputs: in this case Poisson summation and results on integer solutions to binary quadratic forms with positive discriminant. However the proof of Proposition \ref{prop:prop1q} is more intricate than that of Proposition \ref{pr:prop1} due to a number of technical issues.
The proof of Proposition \ref{prop:prop2q} is similar to the proof of Proposition \ref{pr:prop2} and uses hybrid versions of Huxley's large value estimates, subconvexity estimates for $L(s, \chi)$ and a hybrid fourth moment estimate.

The deduction of Theorem \ref{thm:main2} from the above two proposition is identical to the deduction of Theorem \ref{thm:main} from Proposition \ref{pr:prop1} and Proposition \ref{pr:prop2}. The only difference is that we use the result of Nunes to handle the case when $q > x^{1 - \varepsilon}$ and we notice that for prime $q$,
\begin{multline*}
\frac{1}{\varphi(q)} \sum_{\substack{d^2 n \leq x \\ (d^2 n, q) = 1}} \mu(d) = \frac{1}{\varphi(q)} \sum_{\substack{d^2\leq x \\ (d, q) = 1}} \mu(d) \Big(\Big\lfloor \frac{x}{d^2}\Big\rfloor - \Big\lfloor \frac{x}{qd^2}\Big\rfloor \Big) \\
= \frac{x}{q} \sum_{\substack{d^2 \leq x \\ (d,q)=1}} \frac{\mu(d)}{d^2} + O\Big(\frac{\sqrt{x}}{q}\Big) = \frac{6}{\pi^2} \frac{x}{q} \Big(1-\frac{1}{q^2}\Big)^{-1} + O\Big(\frac{\sqrt{x}}{q}\Big) 
\end{multline*}
and the total error term incurred is $x / q^2$ which is $\leq x^{-\varepsilon} \sqrt{x/ q}$ for $q > x^{1/3 + \varepsilon}$.

Theorem \ref{thm:main_upper} depends upon similar principles as Propositions \ref{pr:prop2} and \ref{prop:prop2q}. We prove Theorem \ref{thm:main_upper} in section \ref{sec:GRH_bounds}.

Finally let us make a few remarks on the bottleneck that prevents us from pushing our result further. Taking $H = X^{6/11}$, we are unable to show the following estimate,
$$
\frac{1}{X} \int_{X}^{2X} \Big | \sum_{\substack{x \leq n d^2 \leq x + H \\ d^2 \sim X^{8/11}}} \mu(d) - H \sum_{d^2 \sim X^{8/11}} \frac{\mu(d)}{d^2}\Big |^2 dx \ll_{A} \frac{\sqrt{H}}{\log^{A} X} 
$$
Specifically opening $\mu(d)$ using Heath-Brown's identity (see \cite{HBV}) the only situation that we are not able to estimate is the one in which $\mu(d)$ is replaced by two smooth sums of equal length. Roughly speaking this corresponds to estimating,
$$
\frac{1}{X} \int_{X}^{2X} \Big | \sum_{\substack{x \leq n a^2 b^2 \leq x + H \\ a, b \sim X^{2/11}}} 1 - H \sum_{a, b \sim X^{2/11}} \frac{1}{a^2 b^2} \Big |^2 dx \ll \frac{\sqrt{H}}{\log^{A} X}
$$
Opening the above expression into Dirichlet polynomials this is roughly equivalent to
$$
\int_{|t| \leq X^{5/11}} \Big | \sum_{n \sim X^{3/11}} \frac{1}{n^{1/2 + it}} \sum_{a \sim X^{2/11}} \frac{1}{a^{1/2 + 2 it}} \sum_{b \sim X^{2/11}} \frac{1}{b^{1/2 + 2 it}} \Big |^2 dt \ll \frac{X^{6/11}}{\log^{A} X}
$$
Applying the functional equation on the Dirichlet polynomial over $n$, and setting $Y = X^{12/11}$ we then see that obtaining the above estimate is equivalent to showing that,
$$
\int_{|t| \leq Y^{5/12}} \Big | \sum_{n \sim Y^{1/6}} \frac{1}{n^{1/2 + it}} \sum_{a \sim Y^{1/6}} \frac{1}{a^{1/2 + 2 it}} \sum_{b \sim Y^{1/6}} \frac{1}{b^{1/2 + 2 it}} \Big |^2 dt \ll \frac{Y^{1/2}}{\log^{A} Y}. 
$$ 
If the $2 i t$ in the Dirichlet polynomial over $a$ and $b$ were replaced by $i t$ then we would be facing exactly the same bottleneck as in the case of improving Huxley's prime number theorem in short intervals (by a variant of the computations in \cite{HBV}, see also \cite[Chapter 7]{Harman}). In particular to make further progress we either need to find a way to improve Huxley's estimate or find a way to exploit the fact that the phases in two of the Dirichlet polynomials are $2 it $ and not $it$. Unfortunately we do not see how to make progress on either of these questions. 

\section{Lemmas}

\subsection{Dirichlet polynomials and $L$-functions}
Let us first collect some standard results on large values of Dirichlet polynomials and $L$-functions.
\begin{lemma}[Large-value theorem]
\label{le:LVT}
Let $N, T \geq 1$ and $V > 0$. Let $F(s) = \sum_{n \leq N} a_n n^{-s}$ be a Dirichlet polynomial and let $G = \sum_{n \leq N} |a_n|^2$. Let $\mathcal{T}$ be a set of $1$-spaced points $t_r \in [-T, T]$ such that $|F(it_r)| \geq V$. Then
\[
|\mathcal{T}| \ll (GNV^{-2} + T \min\{GV^{-2}, G^3 N V^{-6}\})(\log 2NT)^6
\]
\end{lemma}

\begin{proof}
This follows from the mean-value theorem and Huxley's large value theorem, see e.g. \cite[Theorem 9.7 and Corollary 9.9]{iwaniec2004analytic}.
\end{proof}

We will say that a set of tuples $(t, \chi)$ with $\chi$ a Dirichlet character and $t$ a real number is \textit{well-spaced} whenever it holds that if $(t, \chi) \neq (u, \chi')$ then either $\chi \neq \chi'$ or $|t - u| \geq 1$. 

\begin{lemma}[Hybrid large-value theorem]
  \label{le:hybridLVT}
  Let $q \in \mathbb{N}$, $N, T \geq 1$ and $V > 0$. Let $F(s, \chi) = \sum_{n \leq N} a_n \chi(n) n^{-s}$ be a Dirichlet polynomial, and let $G = \sum_{n \leq N} |a_n|^2$. Let $\mathcal{T}$ be a set of well-spaced tuples $(t_r, \chi)$ with $t_r \in [-T, T]$ and  with $\chi$ a primitive character of modulus $q$ such that $|F(i t_r, \chi)| \geq V$. Then
  $$
  |\mathcal{T}| \ll (G N V^{-2} + q T \min \{ G V^{-2}, G^3 N V^{-6} \}) \cdot (\log 2qNT)^{18}. 
  $$
\end{lemma}
\begin{proof}
This follows e.g. from \cite[Theorems 9.16 and 9.18 with $k = q$ and $Q = 1$]{iwaniec2004analytic}.
\end{proof}

\begin{lemma}[Fourth moment estimate]
\label{le:ZetaFourth}
  Let $T \geq 2$. Then
  $$
  \int_{|t| \leq T} |\zeta(\tfrac 12 + it)|^4 dt \ll T (\log T)^4. 
  $$
\end{lemma}

\begin{proof} 
See e.g. \cite[formula (7.6.1)]{Titchmarsh86}.
\end{proof}

\begin{lemma} [Hybrid fourth moment estimate]
\label{le:hybridFourth}
  Let $T \geq 2$ and $q \geq 2$. Then
  $$
  \sum_{\chi} \int_{|t| \leq T} |L(\tfrac 12 + it, \chi)|^4 dt \ll T \varphi(q) \log (T q)^4,
  $$
  where the sum is over all characters modulo $q$.
\end{lemma}

\begin{proof}
See e.g. \cite[Theorem 10.1]{Montgomery71}.
\end{proof}

\begin{lemma}[Subconvexity estimate]
\label{le:ZetaSC}
One has, for $|t| \geq 2$,
\[
\zeta(1/2 + it) \ll |t|^{1/6} (\log |t|)^2.
\]
\end{lemma}
\begin{proof} See e.g. \cite[formula (8.22)]{iwaniec2004analytic}.
\end{proof}

\begin{lemma}[Hybrid Weyl subconvexity] \label{le:hybrid}
  For cube-free $q$, primitive characters $\chi \Mod{q}$ and $|t| \geq 2$,
  $$
  L(1/2 + it, \chi) \ll_\varepsilon (q t)^{1/6 + \varepsilon}
  $$
  for any $\varepsilon > 0$.
\end{lemma}
\begin{proof}
  See \cite[Theorem 1.1]{petrow2019fourth}. 
\end{proof}

Of course the Lindel\"of and Generalized Lindel\"of Hypotheses would give us respectively that for any $\varepsilon > 0$, for $|t| \geq 2$ and any character $\chi \Mod q$,
$$
\zeta(1/2+it) \ll |t|^\varepsilon, \quad
L(1/2+it,\chi) \ll (q|t|)^\varepsilon.
$$

\begin{lemma}[Hybrid mean-value theorem] \label{le:hmvt}
  Let $a(n)$ be an arbitrary sequence of coefficients and $N, q \geq 1$ be integers and $T \geq 1$ real. Then, for any given $\varepsilon > 0$, 
  $$
  \sum_{\chi \Mod{q}} \int_{|t| \leq T} \Big | \sum_{n \leq N} a(n) \chi^2(n) n^{it} \Big |^2 dt \ll q^{\varepsilon} (q T + N) \sum_{n \leq N} |a(n)|^2. 
  $$
\end{lemma}
\begin{proof}
  We notice that given a character $\psi \Mod{q}$ there are at most $\ll q^{\varepsilon}$ characters $\chi$ such that $\chi^2 = \psi$. Therefore the left-hand side of the claim is bounded by
  $$
  \ll q^{\varepsilon} \sum_{\psi \Mod{q}} \int_{|t| \leq T} \Big | \sum_{n \leq N} a(n) \psi(n) n^{it} \Big |^2 dt 
  $$
  and the result follows from the standard hybrid mean-value theorem, see e.g. \cite[Theorem 6.4]{Montgomery71}. 
\end{proof}

\subsection{Asymptotic estimates}

\begin{lemma}\label{le:asympt}
  Fix $\varepsilon \in (0, \tfrac{1}{100})$. Let $K_0$ be a positive constant. Suppose that $W \colon \mathbb{R} \rightarrow \mathbb{C}$ is such that, for all $k, \ell \in \{0, 1, 2, 3, 4\}$, one has
	\begin{equation}\label{eq:uniform_schwarz}
	|W^{(k)}(y)| \leq K_0 \frac{H^{\ell \varepsilon/4}}{(1+|y|)^\ell}, \quad\; \textrm{for all}\; y \in \mathbb{R}.
	\end{equation}
	Let $z \geq H^{1+\varepsilon}$. Then
  \begin{equation}
  \label{eq:constant_asymp}
  2 H^2 \sum_{d_1^2, d_2^2 \leq z} \frac{\mu(d_1) \mu(d_2)}{d_1^2 d_2^2} \sum_{\lambda \geq 1} \Big| W \Big ( \frac{H \lambda}{(d_1^2, d_2^2)} \Big ) \Big|^2 = C H^{1/2} \pi \int_0^\infty |W(y)|^2 \sqrt{y} dy + O(H^{1/2-\varepsilon/5}),
  \end{equation}
  where $C$ is as in~\eqref{eq:C_dfn} and where the implied constant depends only on $K_0$ and $\varepsilon$.
\end{lemma}

\begin{proof}
The proof consists of two steps.

The first step is to show that the sum in \eqref{eq:constant_asymp} can be completed into a sum over all $d_1, d_2$ without affecting the claimed asymptotic. We use the information \eqref{eq:uniform_schwarz} for $k=0$ and $\ell = 0, 1$ to see that, for any $\nu > 0$,
\begin{equation}\label{eq:S_bound}
\sum_{\lambda \geq 1} |W(\lambda/\nu)|^2 \ll \sum_{1 \leq \lambda \leq \nu H^{\varepsilon/4}} 1 + \sum_{\lambda > \nu H^{\varepsilon/4}} \frac{H^{\varepsilon/2} \nu^2}{\lambda^2} \ll \nu H^{\varepsilon/4}.
\end{equation}

Hence
$$
2H^2 \sum_{\substack{d_1 > z^{1/2} \\ \textrm{or} \\ d_2 > z^{1/2}}} \frac{\mu(d_1) \mu(d_2)}{d_1^2 d_2^2} \sum_{\lambda\geq 1} \Big|W\Big(\frac{H\lambda}{(d_1^2,d_2^2)}\Big) \Big|^2 \ll H^{1+\varepsilon/4} \sum_{\substack{d_1 > z^{1/2} \\ \textrm{or} \\ d_2 > z^{1/2}}} \frac{(d_1,d_2)^2}{d_1^2 d_2^2}.
$$
Writing $(d_1, d_2) = d_0$ and $d_i = \delta_i d_0$ and utilizing symmetry and the lower bound for $z$, we see that this is
\[
\ll H^{1+\varepsilon/4} \sum_{d_0 \geq 1} \frac{1}{d_0^2} \sum_{\substack{\delta_1 \geq H^{(1+\varepsilon)/2}/d_0 \\ \delta_2 \geq 1}} \frac{1}{\delta_1^2 \delta_2^2} \ll H^{1+\varepsilon/4} \sum_{d_0 \geq 1} \frac{1}{d_0^2} \min\left\{1, \frac{d_0}{H^{(1+\varepsilon)/2}}\right\} \ll H^{1/2-\varepsilon/5}.
\]

Thus the left-hand side of \eqref{eq:constant_asymp} is 
\begin{equation}\label{eq:completed_sum}
2H^2 \sum_{d_1, d_2 \geq 1} \frac{\mu(d_1) \mu(d_2)}{d_1^2 d_2^2} \sum_{\lambda\geq 1} \Big| W\Big(\frac{H\lambda}{(d_1^2,d_2^2)}\Big) \Big|^2 + O(H^{1/2-\varepsilon/5}).
\end{equation}

The second step is to use contour integration to simplify \eqref{eq:completed_sum}. Define $g(x) = |W(e^x)|^2 e^x$, which is smooth and decays exponentially as $|x|\rightarrow\infty$. Now
\[
\hat{g}(\xi) = \int_{-\infty}^\infty |W(e^x)|^2 e^x e(-x\xi) dx = \int_0^\infty |W(y)|^2 y^{-2\pi i \xi} dy,
\] 
and standard partial integration arguments show that (i) $\hat{g}(\xi)$ is entire and (ii) $\hat{g}(\xi) = O(H^{2\varepsilon}/(|\xi|+1)^3)$ uniformly for $|\Im(\xi)| < 1/(2\pi)$. Fourier inversion implies, for $r > 0$,
$$
|W(r)|^2 = r^{-1} \frac{1}{2\pi i} \int_{(c)} r^s \hat{g}\Big(\frac{s}{2\pi i}\Big)\, ds,
$$
where the integral is over $\Re(s) = c$, and $-1 < c < 1$.

Hence, taking $c = -1/4$,
\begin{multline*}
2H^2 \sum_{d_1, d_2 \geq 1} \frac{\mu(d_1) \mu(d_2)}{d_1^2 d_2^2} \sum_{\lambda\geq 1} \Big| W\Big(\frac{H\lambda}{(d_1^2,d_2^2)}\Big)\Big|^2 \\
= \frac{H}{i\pi} \sum_{d_1, d_2} \frac{\mu(d_1)\mu(d_2)}{d_1^2 d_2^2} \sum_{\lambda \geq 1} \frac{(d_1, d_2)^2}{\lambda} \int_{(-1/4)} H^s \lambda^s (d_1, d_2)^{-2s} \hat{g}\Big(\frac{s}{2\pi i}\Big)\, ds.
\end{multline*}
The range of $s$ is such that the sums over both $\lambda$ and $d_1, d_2$ can be taken inside the integral, and the above simplifies to
\begin{multline*}
\frac{H}{i\pi} \int_{(-1/4)} H^s \zeta(1-s) \prod_p \Big(1 - \frac{2}{p^2} + \frac{1}{p^{2+2s}}\Big) \hat{g}\Big(\frac{s}{2\pi i}\Big)\, ds \\
= \frac{H}{i\pi} \int_{(-1/4)} H^s \zeta(1-s) \zeta(2+2s) \prod_p \Big(1 - \frac{2}{p^2} + \frac{2}{p^{4+2s}} - \frac{1}{p^{4+4s}}\Big) \hat{g}\Big(\frac{s}{2\pi i}\Big)\, ds.
\end{multline*}
The Euler product in the last line converges absolutely for $\Re s > -3/4$. Therefore using Lemma \ref{le:ZetaSC} (noting that the same bound holds also for $\zeta(c+it)$ with $c \geq 1/2$) and bounds on $\hat{g}$ we can push the contour integral above to the left to an integral over $\Re(s) = -3/4+\varepsilon$. Because of the singularity from $\zeta(2+2s)$ at $s = -1/2$ the above then simplifies to
\[
\begin{split}
&H^{1/2}\zeta(3/2) \prod_p \Big(1 - \frac{3}{p^2} + \frac{2}{p^3}\Big) \hat{g}\Big(-\frac{1}{4\pi i}\Big) + O(H^{1/4+3\varepsilon}) \\
&= H^{1/2} \int_0^\infty |W(y)|^2 \sqrt{y}\, dy  \zeta(3/2) \prod_p \Big(1 - \frac{3}{p^2} + \frac{2}{p^3}\Big) + O(H^{1/2-\varepsilon/5}).
\end{split}
\]
This verifies the lemma.
\end{proof}

We also have a minor variant:

\begin{lemma}\label{le:asymptS}
Fix $\varepsilon \in (0,\tfrac{1}{100})$. Let $S(x)= \frac{\sin \pi x}{\pi x}$, defined by continuity at $x=0$, and let $z \geq H^{1+\varepsilon}$. Then 
\begin{equation}
  \label{eq:S_constant_asymp}
  2 H^2 \sum_{d_1^2, d_2^2 \leq z} \frac{\mu(d_1) \mu(d_2)}{d_1^2 \cdot d_2^2} \sum_{\lambda \geq 1} S \Big ( \frac{H \lambda}{(d_1^2, d_2^2)} \Big )^2 = C H^{1/2} + O(H^{1/2-\varepsilon/8}).
\end{equation}
\end{lemma}

\begin{proof}
We first note that
\begin{equation}\label{eq:sinc_integral}
\int_0^\infty S(y)^2 \sqrt{y}\, dy = \frac{1}{\pi}.
\end{equation}
This identity follows from \cite[formula 3.823]{gradshteyn2014table}.

Thus \eqref{eq:S_constant_asymp} is a variant of \eqref{eq:constant_asymp}. Lemma \ref{le:asympt} does not apply directly because $S$ does not decay quickly enough. To overcome this issue, we let $h$ be a smooth bump function such that $h(x) = 1$ for $|x| \leq 1$ and $h(x) = 0$ for $|x|\geq 2$. We introduce the function
$$
W(y) = S(y)h(y/H^{\varepsilon/4})
$$
which satisfies the hypothesis of Lemma \ref{le:asympt} for our $\varepsilon$. On the other hand for such $W$
\begin{multline*}
2 H^2 \sum_{d_1^2, d_2^2 \leq z} \frac{\mu(d_1) \mu(d_2)}{d_1^2 \cdot d_2^2} \sum_{\lambda \geq 1} \left( S \Big ( \frac{H \lambda}{(d_1^2, d_2^2)} \Big )^2 - W \Big ( \frac{H \lambda}{(d_1^2, d_2^2)} \Big )^2 \right) \\
\ll H^2 \sum_{d_1^2, d_2^2} \frac{1}{d_1^2 d_2^2} \sum_{\lambda \geq 1} \frac{1}{(H\lambda/(d_1^2,d_2^2))^2} \mathbf{1}\Big( \frac{H\lambda}{(d_1^2,d_2^2)} \geq H^{\varepsilon/4}\Big).
\end{multline*}
We split the sum over $d_1$ and $d_2$ into the complementary ranges $(d_1^2, d_2^2) \leq H^{1-\varepsilon/4}$ and $(d_1^2, d_2^2) > H^{1-\varepsilon/4}$. In the second case we utilize that $\lambda > H^{\varepsilon/4-1} (d_1^2, d_2^2)$, and we see that the above is
$$
\ll \sum_{(d_1,d_2)^2 \leq H^{1-\varepsilon/4}} \frac{(d_1,d_2)^4}{d_1^2 d_2^2} + H^{1-\varepsilon/4}\sum_{(d_1,d_2)^2 > H^{1-\varepsilon/4}} \frac{(d_1,d_2)^2}{d_1^2 d_2^2}.
$$
Writing $(d_1, d_2) = d_0$ and $d_i = \delta_i d_0$, the above is
\begin{equation}\label{eq:crude_quadrat_bound}
\ll \sum_{d_0 \leq H^{1/2-\varepsilon/8}} \sum_{\delta_1, \delta_2} \frac{1}{\delta_1^2\delta_2^2} + H^{1-\varepsilon/4} \sum_{d_0 > H^{1/2-\varepsilon/8}} \frac{1}{d_0^2} \sum_{\delta_1, \delta_2} \frac{1}{\delta_1^2\delta_2^2} 
\ll H^{1/2-\varepsilon/8}.
\end{equation}
On the other hand, 
$$
\int_0^\infty S(y)^2\sqrt{y}\,dy - \int_0^\infty W(y)^2\sqrt{y}\,dy \ll \int_{H^{\varepsilon/4}}^\infty y^{-3/2}\, dy \ll H^{-\varepsilon/8}.
$$
Combining this with the bound \eqref{eq:crude_quadrat_bound} and the identity \eqref{eq:sinc_integral} verifies \eqref{eq:S_constant_asymp} with error term of order $H^{1/2-\varepsilon/5} + H^{1/2-\varepsilon/8} \ll H^{1/2-\varepsilon/8}$.
\end{proof}

\subsection{Initial reductions on second moments}
The following lemma will be used in the proof of Proposition~\ref{pr:prop2}.
\begin{lemma}
\label{le:SV}
  If $F \colon \mathbb{R} \rightarrow \mathbb{C}$ is square-integrable and $H \leq X$, then
  $$
  \int_{X}^{2X} |F(x + H) - F(x)|^2 dx \ll \sup_{\theta \in [\frac{H}{3X}, \frac{3 H}{X}]} \int_{X}^{3X} |F(u + \theta u) - F(u)|^2 du
  $$
\end{lemma}

\begin{proof} The proof can be found in a paper by Saffari and Vaughan~\cite[Page 25]{SaffariVaughan77} but for the convenience of the reader we include the proof here.

First note that  by the triangle inequality we have, for any $v \geq H$,
$$
|F(x+H)-F(x)|^2 \ll |F(x+v)-F(x)|^2 + |F(x+v)-F(x+H)|^2.
$$
Integrating this over $x \in [X,2X]$ and $v \in [2H,3H]$,
\begin{multline*}
H \int_X^{2X} |F(x+H)-F(x)|^2\,dx \ll  \int_{2H}^{3H} \int_X^{2X} |F(x+v)-F(x)|^2\, dx dv \\
+ \int_{2H}^{3H} \int_X^{2X} |F(x+v) - F(x+H)|^2\, dx dv. 
\end{multline*}
By a change of variables the right-hand side is equal to
\begin{multline*}
\int_{2H}^{3H} \int_X^{2X} |F(x+v)-F(x)|^2\, dxdv
+ \int_H^{2H} \int_{X+H}^{2X+H} |F(y+w)-F(y)|^2\, dy dw \\
\le  \int_H^{3H} \int_X^{3X} |F(x+v)-F(x)|^2\, dx dv 
= \int_X^{3X} \int_H^{3H}  |F(x+v)-F(x)|^2\, dv dx.
\end{multline*}
Changing the order of integration was justified by Fubini's theorem. Letting $v = \theta x$ in the inner integral of the last expression above, we see the right-hand side is equal to
$$
\int_X^{3X} \int_{H/x}^{3H/x} |F(x+\theta x) - F(x)|^2 x \, d\theta dx \ll X \int_X^{3X} \int_{H/3X}^{3H/X} |F(x+\theta x) - F(x)|^2\, d\theta dx.
$$
Collecting everything and swapping the order of integration again, we obtain
\[
H \int_X^{2X} |F(x+H) - F(x)|^2\, dx \ll X \int_{H/3X}^{3H/X} \int_X^{3X} |F(u+\theta u) - F(u)|^2\, dud\theta,
\]
which immediately implies the claim.
\end{proof}

We will frequently use the following immediate consequences of the orthogonality of characters: For any sequence $b_n$ of complex numbers,
\begin{equation}
\label{eq:transWOchi0}
\begin{split}
\frac{1}{\varphi(q)} \sum_{\substack{\chi \Mod{q} \\ \chi \neq \chi_0 }} \left| \sum_{n} b_n \chi(n)\right|^2 &= \sum_{\substack{n_1\, \equiv\, n_2 \Mod{q} \\ (n_1 n_2, q) = 1}} b_{n_1} \overline{b_{n_2}} - \frac{1}{\varphi(q)} \sum_{\substack{n_1, n_2 \\ (n_1 n_2, q) = 1}} b_{n_1} \overline{b_{n_2}}\\
&= \sum_{\substack{a \Mod q \\ (a, q) = 1}} \left| \sum_{n\, \equiv\, a \Mod{q}} b_n - \frac{1}{\varphi(q)} \sum_{(n, q) = 1} b_n \right|^2
\end{split}
\end{equation}
and
\begin{equation}
\label{eq:transWchi0}
\begin{split}
\frac{1}{\varphi(q)} \sum_{\chi \Mod q} \left| \sum_{n} b_n \chi(n)\right|^2 &= \sum_{\substack{n_1\, \equiv\, n_2 \Mod{q} \\ (n_1 n_2, q) = 1}} b_{n_1} \overline{b_{n_2}} = \sum_{\substack{a \Mod q \\ (a, q) = 1}} \left| \sum_{n = a \Mod{q}} b_n \right|^2.
\end{split}
\end{equation}

\subsection{Point-counting lemmas}

\begin{lemma} \label{le:lat}
  Let $a,b \in \mathbb{N}$ be such that $\sqrt{b / a}$ is irrational. Let $\eta \in (0, 1]$ and $M \geq 1$. The number of $m \sim M$ such that
  \begin{equation}
  \label{eq:msqrtbaineq}
  \Big \| m \sqrt{\frac{b}{a}} \Big \| \leq \eta
  \end{equation}
  is bounded by
  $$
  \ll \eta M + \sqrt{\eta M} (a b)^{1/4} + 1.
  $$
\end{lemma}
\begin{proof}
We can clearly assume that $\eta^{1/2} (ab)^{1/4} \leq M^{1/2}$ since otherwise the claim is trivial.
Assume we have a (reduced) rational approximation $r/q$ with $r \in \mathbb{Z}$ and $q \in \mathbb{N}$ such that
\begin{equation}
\label{eq:ratapprox}
\left|\sqrt{\frac{b}{a}} - \frac{r}{q}\right| \leq \frac{1}{q^2}.
\end{equation}
Now, writing each $m \in (M, 2M]$ as $m = kq + \ell$ with $0 \leq \ell \leq q-1$, we see that the number of solutions to~\eqref{eq:msqrtbaineq} with $m \sim M$ is at most
\[
\begin{split}
& \sum_{\lfloor M/q \rfloor \leq k \leq 2M/q} \Bigl|\Bigl\{0 \leq \ell \leq q-1 \colon \Big \| (kq + \ell) \sqrt{\frac{b}{a}} \Big \| \leq \eta \Bigr\}\Bigr|\\
&\ll \left(\frac{M}{q} + 1\right) \max_{\xi \in [0, 1]} \Bigl|\Bigl\{0 \leq \ell \leq q-1 \colon \Big \| \ell \sqrt{\frac{b}{a}} + \xi \Big \| \leq \eta \Bigr\}\Bigr| \\
&\ll \left(\frac{M}{q} + 1\right) \max_{\xi \in [0, 1]} \Bigl|\Bigl\{0 \leq \ell \leq q-1 \colon \Big \| \frac{\ell r}{q} + \xi \Big \| \leq \eta + 1/q \Bigr\}\Bigr| \\
&\ll \left(\frac{M}{q} + 1\right) \left(q \cdot \eta + 1\right) \ll M\eta + \frac{M}{q} + q\eta + 1.
\end{split}
\]

Now since $\sqrt{b/a} = \sqrt{ab}/a$ is a quadratic irrational, the partial denominators in its continued fraction expansion have size at most $2\sqrt{ab}$ (see for instance \cite[p.~44]{rockett1992}). In particular this means that for any given $R \geq 1$, we can find $q \in [R, 3\sqrt{ab} R]$ such that~\eqref{eq:ratapprox} holds for some $r$ coprime to $q$. Taking $R = M^{1/2}/(\eta^{1/2} (ab)^{1/4}) \geq 1$, we see that the number of solutions is indeed
\[
\ll M\eta + M^{1/2} \eta^{1/2} (ab)^{1/4} + 1.
\]
  
\end{proof}

\begin{lemma}\label{le:lat2}
Let $a, b \in \mathbb{N}$ be such that $\sqrt{b/a}$ is irrational, and let $M_1, M_2, T \geq 1$. The number of solutions to 
\[ 
|a m_1^2 - b m_2^2| \leq \frac{b M_2^2}{T} \quad \text{with $m_1 \sim M_1$ and $m_2 \sim M_2$}
\]
is
\[
\ll \frac{M_1 M_2}{T} + \Big ( \frac{(M_1 M_2)^{1/2}(ab)^{1/4}}{T^{1/2}} + 1 \Big ) \cdot \mathbf{1}_{M_2 < T}.
\]
\end{lemma}

\begin{proof}
Dividing by $b$ and factoring, we see that we need to count the number of solutions to
\[
\left|\left(m_1 \sqrt{\frac{a}{b}} - m_2\right)\left(m_1 \sqrt{\frac{a}{b}} + m_2\right)\right| \leq \frac{M_2^2}{T}
\]
Dividing by the second factor, we see that it suffices to count the number of solutions to 
\[
\left|m_1 \sqrt{\frac{a}{b}} - m_2\right| \leq \frac{M_2}{T}.
\]
If $M_2 \geq T$, we have $M_1$ choices for $m_1$ and after that $O(M_2/T)$ choices for $m_2$, so in total $M_1 M_2/T$ solutions which is fine.

If $M_2 < T$, then once $m_1$ is chosen there are at most two choices for $m_2$. Therefore it suffices to count the number of integers $m_1 \sim M_1$ such that 
\[
\left\| m_1 \sqrt{\frac{b}{a}} \right\| \leq \frac{M_2}{T}.
\]
The result now follows from Lemma \ref{le:lat}. 
\end{proof}

    \section{The range $H^{1+\varepsilon} \leq z \leq \min\{X/H^{1/2+\varepsilon}, H^{1/2-\varepsilon}X^{1/2}\}$ in the $t$-aspect : Proof of Proposition \ref{pr:prop1}}
\label{se:prop1}
    In what follows we let $S$ be the sinc function as defined in Lemma \ref{le:asymptS}. Proposition~\ref{pr:prop1} follows immediately combining the following proposition with Lemma~\ref{le:asymptS}.
    \begin{proposition}\label{thm:smalldivisors}
Let $X^{\varepsilon} \leq H \leq X^{2/3 - \varepsilon}$ and $H^{1+\varepsilon}  \leq z \leq \min\{X^{1-\varepsilon}/H^{1/2}, H^{1/2-\varepsilon}X^{1/2}\}$. 
Then, as $X \rightarrow \infty$, 
\begin{multline*}
\frac{1}{X} \int_X^{2X} \Big| \sum_{d^2 \leq z} \mu(d) \sum_{x/d^2 < n \leq (x+H)/d^2} 1 - H \sum_{d^2 \leq z} \frac{\mu(d)}{d^2}\Big|^2\, dx \\
= (1+O(H^{-\varepsilon/2}))2H^2 \sum_{k_1^2,k_2^2 \leq z} \frac{\mu(k_1) \mu(k_2)}{k_1^2 k_2^2} \sum_{\lambda \geq 1} S\Big(\frac{H\lambda}{(k_1^2,k_2^2)}\Big)^2 + O(H^{1/2-\varepsilon/3}).
\end{multline*}
\end{proposition}

\begin{proof}

We prove a smoothed version of the claim first. Let $\sigma\colon \mathbb{R}\rightarrow\mathbb{R}$ be an absolutely integrable function such that $\hat{\sigma}$ is supported in the interval $[-BH^{\varepsilon/2},B H^{\varepsilon/2}]$ for some constant $B$ to be specified later. We first show that as $X\rightarrow\infty$,
\begin{multline}
\label{smooth_funceq_var}
\frac{1}{X} \int_{-\infty}^{\infty} \sigma\left(\frac{x}{X}\right) \Big| \sum_{d^2 \leq z} \mu(d) \sum_{x/d^2 < n \leq (x+H)/d^2} 1 - H \sum_{d^2 \leq z} \frac{\mu(d)}{d^2}\Big|^2\, dx \\
= 2\hat{\sigma}(0) H^2 \sum_{k_1^2,k_2^2 \leq z} \frac{\mu(k_1) \mu(k_2)}{k_1^2 k_2^2} \sum_{\lambda \geq 1} S\Big(\frac{H\lambda}{(k_1^2,k_2^2)}\Big)^2 + O(H^{1/2-\varepsilon/3}).
\end{multline}

Here
\begin{equation}
\label{eq:intpsi}
\sum_{x/d^2 < n \leq (x+H)/d^2} 1 = H/d^2 + \psi(x/d^2) - \psi((x+H)/d^2),
\end{equation}
where $\psi(y) = y - [y] - 1/2$ with $[y]$ the integral part of $y$. For $\psi$ we have the Fourier expansion (see e.g. \cite[(4.18)]{iwaniec2004analytic})
\begin{equation}
\label{eq:psiFour}
\psi(y)  = - \frac{1}{2\pi i} \sum_{0 < |n| \leq N} \frac{1}{n} e(y n) + O(\min\{1, 1/(N\Vert y \Vert)\}).
\end{equation}
We take $N = X^{10}$ and plug~\eqref{eq:psiFour} into~\eqref{eq:intpsi}. The arising error term is $O(1/X^5)$ unless $\Vert x/d^2 \Vert < X^{-5}$ or $\Vert (x+H)/d^2 \Vert < X^{-5}$. Given this, it is easy to see that the error term leads to acceptable contribution to the left hand side~\eqref{smooth_funceq_var}.

Hence, the left hand side of~\eqref{smooth_funceq_var} can be replaced by
\[
\frac{1}{4\pi^2 X} \int_{-\infty}^{\infty} \sigma\left(\frac{x}{X}\right) \Big| \sum_{d^2 \leq z} \mu(d) \sum_{0 < |n| \leq N} \frac{1}{n} e\left(\frac{n x}{d^2}\right)\left(1-e\left(\frac{nH}{d^2}\right)\right)\Big|^2\, dx .
\]
Expanding, this equals
\begin{equation}
\label{LHS_Fourier}
\frac{1}{4 \pi^2}\sum_{d_1^2, d_2^2 \leq z} \sum_{0 < |n_1|, |n_2| \leq N} \mu(d_1) \mu(d_2) \frac{1}{n_1 n_2} \left(1-e\left(\frac{n_1 H}{d_1^2}\right)\right) \overline{\left(1-e\left(\frac{n_2H}{d_2^2}\right)\right)} \hat{\sigma}\Big(-X \Big( \frac{n_1}{d_1^2} - \frac{n_2}{d_2^2}\Big)\Big ).
\end{equation}
Owing to the support of $\hat{\sigma}$ this implies that we may restrict the sum in \eqref{LHS_Fourier} to those integers for which
\begin{equation} \label{eq:constraint}
\Big | \frac{n_1}{d_1^2} - \frac{n_2}{d_2^2} \Big | \leq  \frac{BH^{\varepsilon/2}}{X}.
\end{equation}

We consider separately those $(n_1, n_2, d_1, d_2)$ for which $n_1 d_2^2 = n_2 d_1^2$ and those for which this does not hold. In the first case parameterizing solutions in $n_1$ and $n_2$ by $n_1 = \lambda d_1^2/(d_1^2,d_2^2)$ and $n_2 = \lambda d_2^2 /(d_1^2,d_2^2)$ for $\lambda \in \mathbb{Z}\setminus \{0\}$, we obtain 
\[
\frac{\hat{\sigma}(0)}{4\pi^2} \sum_{d_1^2, d_2^2 \leq z} \mu(d_1) \mu(d_2) \sum_{\lambda \neq 0} \frac{(d_1, d_2)^4}{d_1^2 d_2^2 \lambda^2} \left|1-e\left(\frac{\lambda H}{(d_1, d_2)^2}\right)\right|^2 + O\left(\frac{1}{X^5}\right),
\]
where the error term comes from adding $|n_i| > N$ (for which surely $|\lambda| > X^8$). Here
\[
\left|1-e\left(\frac{\lambda H}{(d_1, d_2)^2}\right)\right| = 2\left|\sin\left(\frac{\lambda \pi H}{(d_1^2, d_2^2)}\right)\right|,
\]
so we get the desired main term involving $S(\lambda H/(d_1^2, d_2^2))$.

Therefore it remains to show that the contribution of terms with $n_1d_2^2 \neq n_2 d_1^2$ is negligible. Splitting $n_j$ and $d_j$ dyadically, we need to bound, for any  $D_1, D_2 \leq z^{1/2}$ and any $N_1, N_2 \leq N$,
\begin{equation} \label{eq:tobound}
\min\left\{\frac{1}{N_1}, \frac{H}{D_1^2}\right\} \min\left\{\frac{1}{N_2}, \frac{H}{D_2^2}\right\} \sum_{\substack{n_1 \sim N_1 \\ n_2 \sim N_2}} \# \left\{(d_1, d_2) \colon d_j \sim D_j, 0 < \Big | \frac{n_1}{d_1^2} - \frac{n_2}{d_2^2} \Big | \leq \frac{BH^{\varepsilon/2}}{X} \right\} 
\end{equation}
and we need a bound that is $O(H^{1/2 - \varepsilon/2})$.
Now 
\begin{equation}
\label{eq:solcount}
\begin{split}
&\# \{(d_1, d_2) \colon d_j \sim D_j, 0 < \Big | \frac{n_1}{d_1^2} - \frac{n_2}{d_2^2} \Big | \leq \frac{BH^{\varepsilon/2}}{X} \} \\
&\ll \# \left\{(d_1, d_2) \colon d_j \sim D_j, 0 < \Big | n_1 d_2^2 - n_2 d_1^2 \Big | \leq 16\frac{BD_2^2 H^{\varepsilon/2}}{X N_2} \cdot D_1^2 N_2\right\}.
\end{split}
\end{equation}
Notice that there are no solutions unless
\begin{equation}
\label{eq:NjDjrel}
N_1 D_2^2 \asymp N_2 D_1^2.
\end{equation}

We split into two cases according to whether $\sqrt{n_2/n_1}$ is quadratic irrational or instead rational. In the first case we can apply Lemma~\ref{le:lat2}, which shows that the number of solutions~\eqref{eq:solcount} is
\[
\begin{split}
& \ll \frac{H^{\varepsilon/2} D_1 D_2^3}{X N_2} + 1+ \frac{D_1^{1/2} D_2^{3/2} N_1^{1/4} H^{\varepsilon/4}}{X^{1/2} N_2^{1/4}}
\end{split}
\]
By~\eqref{eq:NjDjrel} we can multiply the first term by $(D_1/D_2) (N_2/N_1)^{1/2}$ and the third term by $(D_1/D_2)^{1/2} (N_2/N_1)^{1/4}$ to obtain
\[
\ll \frac{H^{\varepsilon/2} D_1^2 D_2^2}{X (N_1N_2)^{1/2}} + 1+ \frac{D_1 D_2 H^{\varepsilon/4}}{X^{1/2}}.
\]

Using this bound in~\eqref{eq:tobound}, and summing over $n_1$ and $n_2$, we note that the maximum is attained for $N_j = D_j^2/H$ and thus the contribution to \eqref{eq:tobound} from $\sqrt{n_2/n_1}$ quadratic irrational is bounded by
\[
\ll H^{\varepsilon/2}\left(\frac{D_1 D_2 H}{X} + 1 +\frac{D_1 D_2}{X^{1/2}} \right) = O(H^{1/2-\varepsilon/2})
\]
since $D_1 \cdot D_2 \leq z \leq \min\{X/H^{1/2+\varepsilon}, H^{1/2-\varepsilon} X^{1/2}\}$.

In case $\sqrt{n_2/n_1}$ is rational, there exist $m, \ell_1, \ell_2 \in \mathbb{Z}$ such that $n_1 = m \ell_1^2$ and $n_2 = m \ell_2^2$. Hence, writing $r_1^2 = \ell_1^2 d_2^2$ and $r_2^2 = \ell_2^2 d_1^2$, we see that the contribution to \eqref{eq:tobound} for $\sqrt{n_2/n_1}$ rational is bounded by
\begin{equation}
\label{eq:nonirrbound}
\begin{split}
&\ll H^{\varepsilon/1000} \min\left\{\frac{1}{N_1}, \frac{H}{D_1^2}\right\} \min\left\{\frac{1}{N_2}, \frac{H}{D_2^2}\right\} \\
&\qquad \times \sum_{m} \#\left\{(r_1, r_2) \colon r_j \leq D_j\sqrt{N_j/m}, 0 < |r_1^2 - r_2^2| \leq \frac{BH^{\varepsilon/2} D_1^2 D_2^2}{m X}\right\} 
\end{split}
\end{equation}
Factoring $r_1^2-r_2^2 = (r_1-r_2)(r_1+r_2)$ and dividing by the second factor, we see that the number of solutions $(r_1, r_2)$ is
\[
\ll \frac{BH^{\varepsilon/2} D_1^2 D_2^2}{m X} \log X
\]
Summing over $m \ll \min\{N_1, N_2\}$ and using this bound in~\eqref{eq:nonirrbound}, the maximum in the resulting bound for ~\eqref{eq:nonirrbound} is attained for $N_j = D_j^2/H$. Hence we obtain that~\eqref{eq:nonirrbound} is at most $H^{2+\varepsilon/2+\varepsilon/500}/X \leq H^{1/2-\varepsilon/2}$ since $H \leq X^{2/3-\varepsilon}$.

Let us now dispose of the smoothing $\sigma$: Take $B$ to be a sufficiently large absolute constant that there exist integrable functions $\sigma_-$ and $\sigma_+$ such that $\widehat{\sigma}_-$ and $\widehat{\sigma}_+$ have support $[-B H^{\varepsilon/2}, BH^{\varepsilon/2}]$, and
$$
\sigma_- \leq \mathbf{1}_{[1,2]} \leq \sigma_+, \quad\textrm{and}\quad \Big| \int \sigma_{\pm}(x)\,dx -1 \Big| \leq H^{-\varepsilon/2}.
$$
(We allow $\sigma_-$ and $\sigma_+$ to take negative values.) An explicit construction of such functions is given by the Beurling-Selberg majorant and minorant \cite[p.~273]{montgomery2001}. Applying \eqref{smooth_funceq_var} and these bounds,
\begin{multline*}
\frac{1}{X} \int_{-\infty}^\infty \mathbf{1}_{[1,2]}\Big(\frac{x}{X}\Big) \Big| \sum_{d^2 \leq z} \mu(d) \sum_{x/d^2 \leq n \leq (x+H)/d^2} 1 - H \sum_{k^2 \leq z} \frac{\mu(k)}{k^2}\Big|^2\, dx \\
= (1+O(H^{-\varepsilon/2})) 2H^2 \sum_{k_1^2,k_2^2 \leq z} \frac{\mu(k_1) \mu(k_2)}{k_1^2 k_2^2} \sum_{\lambda \geq 1} S\Big(\frac{H\lambda}{(k_1^2,k_2^2)}\Big)^2 + O(H^{1/2-\varepsilon/3}).
\end{multline*}
\end{proof}

\section{The range $z \geq H^{4/3+\varepsilon}$ in the $t$-aspect : Proof of Proposition \ref{pr:prop2}}
\label{se:prop2}
We would like to establish that
$$
\frac{1}{X} \int_{X}^{2X} \Big | \sum_{\substack{x < n d^2 \leq x + H \\ d^2 > z}} \mu(d) - H \sum_{\substack{z < d^2 \leq 2 X}} \frac{\mu(d)}{d^2} \Big |^2 dx \ll H^{1/2 - \varepsilon / 8}. 
$$
Splitting into dyadic ranges according to the size of $d$, 
it essentially suffices to show that, for each $D \in [z^{1/2}, (2X)^{1/2}]$, we have
\begin{equation}
\label{eq:varDlar}
\frac{1}{X} \int_{X}^{2X} \Big | \sum_{\substack{x < n d^2 \leq x + H \\ d \sim D}} \mu(d) - H \sum_{d \sim D} \frac{\mu(d)}{d^2} \Big |^2 dx \ll H^{1/2 - \varepsilon / 4}. 
\end{equation}

Let
$$
A(x) := \sum_{\substack{n d^2 \leq x \\ d \sim D}} \mu(d) - x \sum_{d \sim D} \frac{\mu(d)}{d^2}. 
$$
Using this definition and Lemma~\ref{le:SV}, we see that the left-hand side of~\eqref{eq:varDlar} is
\begin{equation} \label{eq:saffari}
\frac{1}{X} \int_{X}^{2X} |A (x + H) - A(x) |^2 dx \ll \frac{1}{X} \int_{X}^{3X} |A(u ( 1+ \theta)) - A(u)|^2 du
\end{equation}
for some $\theta \in [\frac{H}{3X}, \frac{3 H}{X}]$. Choose $w$ such that $e^w = 1 + \theta$, so that $w \asymp \frac{H}{X}$. 
By contour integration
\begin{equation}
\label{eq:Aey}
A(e^y) = \frac{1}{2\pi i} \int_{2-i\infty}^{2+i\infty} \frac{e^{y s}}{s} \zeta(s) M(2s) ds - e^y \sum_{d \sim D} \frac{\mu(d)}{d^2},
\end{equation}
where
$$
M(s) := \sum_{d \sim D} \frac{\mu(d)}{d^{s}}.
$$
Moving the contour to the line $\Re s= 1/2$ we notice that the residue from $s = 1$ cancels with the second term on the right-hand side of~\eqref{eq:Aey}, and we obtain
$$
\frac{A(e^{w + x}) - A(e^x)}{e^{x / 2}} = \frac{1}{2\pi} \int_{\mathbb{R}} \frac{e^{w (\tfrac 12 + it)} - 1}{\tfrac 12 + it} e^{i t x} \zeta(\tfrac 12 + it) M(1 + 2it) dt.
$$
Therefore, by Plancherel,
\begin{equation} \label{eq:plancherel}
\int_{0}^{\infty} | A(e^{u + w}) - A(e^u) |^2 \cdot \frac{du}{e^u} \ll \int_{\mathbb{R}} \Big | \frac{e^{w (\tfrac 12 + it)} - 1}{\tfrac 12 + it} \Big |^2 \cdot |\zeta(\tfrac 12 + it) M(1 + 2it)|^2 dt. 
\end{equation}
Combining \eqref{eq:saffari} and \eqref{eq:plancherel} we get after a change of variable, 
\begin{equation}
\label{eq:Asqdiff}
\begin{split}
  \frac{1}{X} \int_{X}^{2X} |A(x + H) - A(x)|^2 dx & \ll X \int_{0}^{\infty}  |A( u (1 + \theta) ) - A(u) |^2 \frac{du}{u^2} \\
  & \ll X \int_{\mathbb{R}} \Big | \frac{e^{w (\tfrac 12 + it)} - 1}{\tfrac 12 + it} \Big |^2 \cdot |\zeta(\tfrac 12 + it) M(1 + 2it)|^2 dt \\
  & \ll X \int_{\mathbb{R}} \min\Bigl\{ \Bigl(\frac{H}{X}\Bigr)^2, \frac{1}{|t|^2}\Bigr\} \cdot |\zeta(\tfrac 12 + it) M(1 + 2it)|^2 dt.
\end{split}
\end{equation}
By Lemma~\ref{le:ZetaSC} the part with $|t| \geq X^2$ contributes
\[
\ll X \int_{X^2}^\infty |t|^{-5/3+\varepsilon} dt = O(1).
\]
On the other hand, the contribution of $|t|\le X^2$ to the right-hand side of \eqref{eq:Asqdiff} is at most
\begin{align}
\nonumber
&\ll    \frac{H^2}{X}	\int_{|t| \le 2X/H} |\zeta(\tfrac 12 + it) M(1 + 2it)|^2 dt \\
\nonumber
& \qquad \qquad + X \int_{X/H}^{X^2} \frac{1}{T^2} \cdot \frac{1}{T} \int_{T \le |t|\le 2T} |\zeta(\tfrac 12 + it) M(1 + 2it)|^2 dt dT\\
\label{eq:final}
&\ll H \Big ( \sup_{X / H \leq T \leq X^2} \frac{1}{T} \int_{|t| \leq T} |\zeta(\tfrac 12 + it) M(1 + 2it)|^2 dt \Big ) + O(1).
\end{align}

Let us now prove the claim on the assumption of the Lindel\"of Hypothesis. Applying Lindel\"of and then the mean-value theorem (Lemma \ref{le:hmvt} with $q = 1$), we have for any choice of $\delta > 0$,
\[
\begin{split}
\frac{H}{T} \int_{|t|\leq T} |\zeta(1/2+it) M(1+i2t)|^2\, dt &\ll \frac{H T^\delta}{T}\int_{|t|\leq T}|M(1+i2t)|^2\,dt \\ 
&\ll \frac{H T^\delta}{T} (T+D)\cdot \frac{1}{D} \ll \frac{HT^{\delta}}{D} + \frac{HT^\delta}{T}.
\end{split}
\]
Recall we have $D \geq z^{1/2} \geq H^{(1+\varepsilon)/2}, \, H \leq X^{2/3-\varepsilon}$ and $X/H \leq T \leq X^2$. Hence the above is
$$
\ll H^{1/2-\varepsilon/2} T^{\delta} + \frac{H^{2-\delta}}{X^{1-\delta}} \ll H^{1/2 - \varepsilon/4} + \frac{H^{2-\delta}}{H^{(1-\delta)/(2/3-\varepsilon)}} \ll H^{1/2-\varepsilon/4},
$$
for $\delta$ sufficiently small. Applying this bound to \eqref{eq:final} yields the claim.

Let us now prove the unconditional part of the proposition. 
First notice that the values of $t$ for which $|M(1+2it)| \leq D^{-1/2 + \varepsilon/16}$ contribute to~\eqref{eq:final} by Cauchy-Schwarz and the fourth moment bound (Lemma~\ref{le:ZetaFourth}) $O(H^{1+\varepsilon/16}D^{-1+\varepsilon/8}) = O(H^{1/2-\varepsilon/4})$, and therefore their contribution is always acceptable.
Writing 
\[
S(V) = \{t \in [-T, T] \colon V \leq |M(1+2it)| < 2V\},
\]
by dyadic splitting, it suffices to show that, for each $V \in [D^{-1/2}, 1]$ and $T \in [X/H, X^2]$, we have
\[
\frac{H}{T} V^2 \int_{S(V)} |\zeta(1/2+it)|^2 dt \ll H^{1/2-\varepsilon/3}.
\]
Now by Lemma~\ref{le:LVT} we have
\begin{equation}
\label{eq:S(V)bound}
|S(V)| \ll (V^{-2} + T\min\{D^{-1}V^{-2}, D^{-2} V^{-6}\}) (\log 2X)^6.
\end{equation}

Consider first the case when the first term dominates here. Then by Lemma~\ref{le:ZetaSC} we have
\[
\frac{H}{T} V^2 \int_{S(V)} |\zeta(1/2+it)|^2 dt \ll \frac{H}{T} T^{1/3+\varepsilon/2} \ll \frac{H}{T^{2/3-\varepsilon/2}} \ll \frac{H^{5/3}}{X^{2/3-\varepsilon/2}} \leq H^{1/2-\varepsilon/3}
\] 
since $H \leq X^{4/7-\varepsilon}$. 

Consider now the case that the second term dominates in~\eqref{eq:S(V)bound}. Then, by Cauchy-Schwarz and the fourth moment estimate (Lemma~\ref{le:ZetaFourth}),
\begin{align*}
&\frac{H}{T} V^2 \int_{S(V)} |\zeta(1/2+it)|^2 dt \ll \frac{H}{T} V^2 |S(V)|^{1/2} \left(\int_{|t| \leq T} |\zeta(\tfrac 12 + it)|^4 dt \right )^{1/2} \\
& \ll H V^2 \min\{D^{-1}V^{-2}, D^{-2} V^{-6}\}^{1/2} (\log 2X)^5 \ll H\min\{D^{-1/2} V, D^{-1} V^{-1}\} (\log 2X)^5 \\
&\ll H (D^{-1/2} V)^{1/2} (D^{-1} V^{-1})^{1/2} (\log 2X)^5 \ll H D^{-3/4} (\log 2X)^5 \ll H z^{-3/8} (\log 2X)^5 \ll H^{1/2-\varepsilon/3}
\end{align*}
since $z \geq H^{4/3+\varepsilon}$. This finishes the proof of Proposition~\ref{pr:prop2}.

\section{The range $(x/q)^{1+\varepsilon} \leq z < x^{-\varepsilon} \sqrt{qx}$ in the $q$-aspect : Proof of Proposition \ref{prop:prop1q}}
\label{se:prop1q} 
By~\eqref{eq:transWOchi0} Proposition~\ref{prop:prop1q} follows immediately from the following proposition.
\begin{proposition} \label{prop:asymptqaspect}
    Let $\varepsilon \in (0, 1/100)$. Let $q$ be prime with $x^{1/3 + 30 \varepsilon} \leq q \leq x^{1 - \varepsilon}$ and let $(x / q)^{1 + \varepsilon} \leq z \leq x^{-\varepsilon} \sqrt{q x}$. Then 
    \begin{equation} \label{eq:toeval}
    \frac{1}{\varphi(q)} \sum_{\substack{\chi \Mod{q} \\ \chi \neq \chi_0 }} \Big | \sum_{\substack{d^2 \leq z \\ n d^2 \leq x}} \mu(d) \chi(d^2) \chi(n) \Big |^2 = C \sqrt{qx} + O((x/q)^{-\varepsilon/16} \sqrt{qx})
    \end{equation}
    with $C$ as in~\eqref{eq:C_dfn}.
  \end{proposition}

  The proof of Proposition \ref{prop:asymptqaspect} is based on two Propositions that we now describe. Proposition \ref{prop:smoothing} below will be used to introduce
  a smoothing into \eqref{eq:toeval}. Note that it gives an upper bound that is $o(\sqrt{qx})$ whenever $z = o(\sqrt{qx}/(\log x)^{6})$ and the interval $I$ has length $o(x/(\log x)^{12})$.

\begin{proposition} \label{prop:smoothing}
  Let $q$ be prime with $q \leq x,$ let $z \leq x$. 
  Let $I\subset [1,2 x]$ be an interval. Then 
  \begin{equation}
  \label{eq:Istat}
  \frac{1}{\varphi(q)} \sum_{\substack{\chi \Mod{q} \\ \chi \neq \chi_0 }} \Big | \sum_{\substack{d^2 \leq z \\ n d^2 \in I}}\mu(d) \chi(d^2) \chi(n)  \Big |^2 \ll (\log x)^6 \cdot \Big (z + \sqrt{|I| q} \Big ). 
  \end{equation}
\end{proposition}

We will use the following proposition to evaluate the smoothed analogue of \eqref{eq:toeval}.

  \begin{proposition} \label{prop:asympt}
    Let $\varepsilon > 0$ be given. Let $f$ be a smooth function such that $f$ is compactly supported on $[0,1]$ and $f(u) = 1$ for $(x/q)^{-\varepsilon/4} \leq u \leq 1 - (x/q)^{-\varepsilon/4}$ and for each integer $k \geq 0$, we have $f^{(k)}(u) \ll (x/q)^{\varepsilon k/4}$. Let $(x / q)^{1 + \varepsilon} \leq z \leq x^{-\varepsilon} \sqrt{qx}$. Then for $x^{1/3 + 30 \varepsilon} \leq q \leq x^{1 - \varepsilon}$,
    $$
    \frac{1}{\varphi(q)} \sum_{\substack{\chi \Mod{q} \\ \chi \neq \chi_0 }} \Big | \sum_{\substack{ n \geq 1 \\ d^2 \leq z}} f \Big ( \frac{n d^2}{x} \Big ) \mu(d) \chi(d^2) \chi(n) \Big |^2 = C \sqrt{q x} + O((x/q)^{-\varepsilon/10} \sqrt{qx}),
    $$
    where $C$ is as in~\eqref{eq:C_dfn}.
  \end{proposition}

One way to construct $f$ satisfying the assumptions of the proposition is to take $\phi(t)$ to be a smooth function which vanishes for negative $t$ and has $\phi(t)=1$ for $t$ greater than $1$, and then set $f(u) = \phi((x/q)^{\varepsilon/4} u) \phi((x/q)^{\varepsilon/4} (1-u))$.

With these two propositions at hand we are ready to prove Proposition \ref{prop:asymptqaspect}. 
 
  \subsection{Proof of Proposition \ref{prop:asymptqaspect}}
For $m \in \mathbb{N}$, set
$$
A_m := \sum_{\substack{d^2 \mid m \\ d^2 \leq z}} \mu(d),
$$
and let $f$ be as described below Proposition~\ref{prop:asympt}.
Then
$$
\frac{1}{\varphi(q)}\sum_{\substack{\chi \Mod{q} \\ \chi \neq \chi_0 }} \Big | \sum_{\substack{n \leq x}} A_n \chi(n)\Big |^2  = S_1 + O(\sqrt{S_1 S_2} + S_2),
$$
where
$$
S_1 := \frac{1}{\varphi(q)}\sum_{\substack{\chi \Mod{q} \\ \chi \neq \chi_0 }}  \Big | \sum_{n} A_n\chi(n) f \Big (\frac{n}{x} \Big ) \Big |^2 
$$
and
$$
S_2 :=  \frac{1}{\varphi(q)}\sum_{\substack{\chi \Mod{q} \\ \chi \neq \chi_0 }} \Big | \sum_{\substack{n \in I}} A_n \chi(n) \Big (1 - f \Big (\frac{n}{x} \Big )\Big )\Big |^2 
$$
where $I = I_1 \cup I_2$ with $I_1 := [1, x\cdot (x/q)^{ - \varepsilon/4}]$ and $I_2 := [x - x\cdot (x/q)^{ - \varepsilon/4}, x]$.

For $i=1,2$, define
$$
B_i(\chi;t) = \sum_{\substack{n \in I_i \\  n < t}} A_n \chi(n).
$$

By partial summation,
\begin{align*}
\sum_{n \in I_2} A_n \chi(n) \Big (1 - f \Big ( \frac{n}{x} \Big ) \Big ) & = \int_{I_2} \Big ( 1 - f \Big ( \frac{t}{x} \Big )\Big )  d B_2(\chi; t) 
\\ & = \frac{1}{x} \int_{I_2} f' \Big ( \frac{t}{x} \Big ) B_2(\chi; t) dt + B_2(\chi; x).
\end{align*}
Hence, 
\begin{align}
\nonumber &\frac{1}{\varphi(q)} \sum_{\chi \neq \chi_0} \Big | \sum_{n \in I_2} A_n \chi(n) \Big (1 - f \Big ( \frac{n}{x} \Big ) \Big ) \Big |^2 \\
\nonumber & \ll \frac{1}{\varphi(q)} \sum_{\chi \neq \chi_0} \Big | \frac{1}{x} \int_{I_2} f' \Big ( \frac{t}{x} \Big ) B_2(\chi; t) dt \Big |^2 + \frac{1}{\varphi(q)} \sum_{\chi \neq \chi_0} \Big|B_2(\chi; x)|^2 \\ \label{eq:toboudd2} & \leq (x/q)^{\varepsilon / 4} \cdot \frac{1}{\varphi(q)} \sum_{\chi \neq \chi_{0}} \frac{1}{x} \int_{I_2} |B_2(\chi; t)|^2 dt  + \frac{1}{\varphi(q)} \sum_{\chi \neq \chi_0} \Big|B_2(\chi; x)|^2.
\end{align}
Now by Proposition \ref{prop:smoothing} we have, for $t \in I_2$,
$$
\frac{1}{\varphi(q)} \sum_{\chi \neq \chi_{0}} |B_{2}(\chi; t)|^2 \ll (\log x)^6 \cdot \Big ( x^{-\varepsilon} \sqrt{q x} + \sqrt{(t - (x - x\cdot (x/q)^{ - \varepsilon/4})) \cdot q} \Big ). 
$$
Therefore \eqref{eq:toboudd2} is
\begin{multline*}
\ll (\log x)^{6} \cdot (x/q)^{\varepsilon / 4} \cdot \frac{1}{x} \cdot \Big ( x\cdot (x/q)^{ - \varepsilon/4} \cdot x^{-\varepsilon} \sqrt{q x} + x^{3/2} (x/q)^{- 3\varepsilon / 8} \sqrt{q} \Big )\\ + (\log x)^6 (x/q)^{-\varepsilon/8} \sqrt{qx} \ll (\log x)^6 (x/q)^{- \varepsilon / 8} \cdot \sqrt{q x}. 
\end{multline*}

A similar argument shows that 
$$
\frac{1}{\varphi(q)} \sum_{\chi \neq \chi_0} \Big | \sum_{n \in I_1} A_n \chi(n) \Big (1 - f \Big ( \frac{n}{x} \Big ) \Big ) \Big |^2 \ll (\log x)^6 (x/q)^{- \varepsilon / 8} \cdot \sqrt{q x}. 
$$
as well. By $(a+b)^2 \ll |a|^2+|b|^2$ we conclude that
$$
S_2  \ll  (\log x)^6 (x/q)^{-\varepsilon / 8} \sqrt{qx}
$$
as needed. 
On the other hand we can compute $S_1$ by using Proposition \ref{prop:asympt} and this yields the claimed estimate.

\subsection{Proof of Proposition \ref{prop:smoothing}}

    By P\'olya's formula (see \cite[Lemma 1]{vaughan1977exponential}) for $I = [a,b]$ and any $\chi \neq \chi_0$ of modulus $q$,
    $$
    \sum_{n \in I / d^2} \chi(n) = \frac{\tau(\chi)}{2\pi i} \sum_{1 \leq |n| \leq q} \overline{\chi}(n) f_{I / d^2}(n) + O(\log q),
    $$
    where
    $$
    f_{I / d^2}(n) = \frac{1}{n} \cdot \Big ( e \Big ( \frac{n a}{d^2 q} \Big ) - e \Big ( \frac{n b}{d^2 q } \Big ) \Big )  \ll g_{I / d^2}(n) := \begin{cases}
      \frac{|I|}{d^2 q} & \text{ if } |n| \leq \frac{d^2 q}{|I|}, \\
      \frac{1}{n} & \text{ otherwise .}
    \end{cases}
    $$
    We split $d$ and $n$ into dyadic intervals and bound the left-hand side of~\eqref{eq:Istat} by
    \begin{equation}
    \label{eq:PolBound}
    (\log x)^2 \sup_{\substack{D \leq z^{1/2} \\ 1 \leq N \leq q}} \sum_{\chi \Mod q}  \Big | \sum_{d \sim D} \mu(d) \chi(d^2) \sum_{n \sim N} \overline{\chi}(n) f_{I / d^2} (n) \Big |^2 + O(z (\log q)^2). 
    \end{equation}
    The error term is clearly acceptable. We bound the main term of~\eqref{eq:PolBound} using a majorant principle --- by going through the first equality in~\eqref{eq:transWchi0} we can replace coefficients $\mu(d)$ and $f_{I/d^2}(n)$ by their majorants. Hence we get the bound
    $$
    \ll (\log x)^2 \sup_{\substack{D \leq z^{1/2} \\ 1 \leq N \leq q}} \sum_{\chi} \Big | \sum_{d} \chi^2(d) V \Big ( \frac{d}{D} \Big ) \cdot \sum_{n} \chi(n) V(n/N) g_{I / D^2} (N) \Big |^2 
    $$
    with $V$ a smooth function supported on $[1/2, 4]$.
    
    The contribution of the principal character and quadratic character is
    $
    \ll z (\log x)^4
    $
    which is acceptable.
    On the remaining non-principal and non-quadratic characters we apply 
    Cauchy-Schwarz giving the upper bound
    \begin{align}\label{eq:cauchy dyadic}
    \ll (\log x)^2 \sup_{\substack{D \leq z^{1/2} \\ 1 \leq N \leq q}} g_{I/D^2}(N)^2 \Big ( \sum_{\substack{\chi^2 \neq \chi_0}} \Big | & \sum_{n} \chi^2(n) V \Big ( \frac{n}{D}\Big )  \Big |^4 \Big )^{1/2}  \Big ( \sum_{\chi \neq \chi_0} \Big | \sum_{n} \chi(n) V \Big ( \frac{n}{N}\Big )  \Big |^4 \Big )^{1/2}.
  \end{align}
We claim that
\begin{equation}\label{eq:two fourth moments}
\sum_{\substack{\chi^2 \neq \chi_0}} \Big | \sum_{n} \chi^2(n) V \Big ( \frac{n}{D} \Big ) \Big |^4 \ll q D^2 \cdot (\log x)^4 \quad \text{ and }\sum_{\substack{\chi \neq \chi_0}} \Big | \sum_{n} \chi(n) V \Big ( \frac{n}{N} \Big ) \Big |^4 \ll q N^2 \cdot (\log x)^4 .
\end{equation}
We explain the second bound in \eqref{eq:two fourth moments}; the first bound is similar. Let $\widetilde{V}$ be the Mellin transform of $V$. Using contour integration, the decay of $\widetilde{V}$, and H\"older, we get, for every $A \geq 1$,
\begin{align*}
	\sum_{\chi \neq \chi_0} \Big | \sum_{n} \chi(n) V \Big ( \frac{n}{N} \Big ) \Big |^4 &=\sum_{\chi \neq \chi_0} \Big |
	 \int_{\mathbb{R}} L(1/2+it,\chi)\widetilde{V}(1/2+it)  N^{1/2+it}\,dt \Big |^4\\
	 & \ll_A N^2 \sum_{\chi \neq \chi_0} \Big(
	 \int_{\mathbb{R}} \Big |L(1/2+it,\chi)\Big | (1+|t|)^{-A}  dt \Big )^4 \\
	 & \ll_A N^2 \sum_{\chi \neq \chi_0}
	 \int_{\mathbb{R}} \Big |L(1/2+it,\chi)\Big |^4 (1+|t|)^{-A}  \,dt.
\end{align*}
A dyadic decomposition of the integration range and the fourth moment bound for Dirichlet $L$-functions (Lemma~\ref{le:hybridFourth}) yield the second part of \eqref{eq:two fourth moments}.

Using \eqref{eq:two fourth moments} in \eqref{eq:cauchy dyadic}, we obtain an upper bound
\[
\begin{split}
&\ll (\log x)^6 \sup_{\substack{D \leq z^{1/2} \\ 1 \leq N \leq q}} g_{I/D^2}(N)^2 q DN \\
& \ll (\log x)^6 \sup_{\substack{D \leq \sqrt{|I|/q} \\ 1 \leq N \leq q}} g_{I/D^2}(N)^2 q DN + (\log x)^6 \sup_{\substack{\sqrt{|I|/q} < D \leq z^{1/2} \\ 1 \leq N \leq q}} g_{I/D^2}(N)^2 q DN .
\end{split}
\]
Recalling the definition of $g_{I/D^2}(N)$, we see that on the last line, the first $N$-supremum is attained for $N=1$ and the second $N$-supremum is attained for $N = D^2q/|I|$, and we get the bound
\[
\ll (\log x)^6 \sup_{\substack{D \leq \sqrt{|I|/q}}} q D + (\log x)^6 \sup_{\substack{\sqrt{|I|/q} < D \leq z^{1/2}}} |I|/D \ll (\log x)^6 \sqrt{|I| q}
\]
and the claim follows. 

\subsection{Proof of Proposition \ref{prop:asympt}}

    We apply Poisson summation (see e.g.~\cite[formula (4.26)]{iwaniec2004analytic}) in the sum over $n$, getting
    $$
    \sum_{n} \chi(n) f \Big ( \frac{n d^2}{x} \Big ) = \tau(\chi) \cdot \frac{x }{q d^2} \sum_{\ell} \overline{\chi}(\ell) \hat{f} \Big ( \frac{x \ell}{d^2 q} \Big ).
    $$
    Therefore we have to asymptotically estimate 
    \begin{align} \label{eq:first}
     \frac{q}{\varphi(q)} & \cdot \frac{x^2}{q^2} \sum_{\chi \neq \chi_{0}} \Big | \sum_{\substack{d^2 \leq z \\ \ell \in \mathbb{Z}}} \frac{\mu(d)}{d^2} \chi(d^2) \overline{\chi}(\ell) \hat{f} \Big ( \frac{x \ell}{d^2 q} \Big ) \Big |^2  = \\   
  \label{eq:second} & \frac{x^2}{q}  \sum_{\substack{n_1, n_2 \in \mathbb{Z} \\ d_1^2, d_2^2 \leq z \\ d_1^2 n_1\, \equiv\, d_2^2 n_2  \Mod{q} \\ (d_1 d_2 n_1 n_2 , q) = 1}} \frac{\mu(d_1)}{d_1^2} \frac{\mu(d_2)}{d_2^2} \cdot \hat{f} \Big ( \frac{x n_2}{d_1^2 q} \Big ) \overline{\hat{f} \Big ( \frac{x n_1}{d_2^2 q} \Big )} + O ( z x^{2\varepsilon/3} ), 
    \end{align}
    and where $O(z x^{2\varepsilon/3})$ comes from the principal character and from replacing $\varphi(q)$ by $q$. We note that since $z \leq x^{-\varepsilon} \sqrt{qx}$ this contribution is acceptable. 
    Notice that we can add and remove the restrictions
    $d_1, d_2 > x^{1/2 - \varepsilon/6} / \sqrt{q}$ and $|n_1| , |n_2| \leq x^{\varepsilon/3} \cdot z q / x$
    at will because they cost us a negligible error term that is $\ll_{A} x^{-A}$ for any given $A > 0$. 
    Moreover note that $n_1$ and $n_2$ now traverse all of $\mathbb{Z}$. 
    
   We now separate the set of tuples $(n_1, n_2)$ into
  $$
  \mathcal{M} := \{ (k_1^2 m, k_2^2 m) : m \in \mathbb{Z} \text{ squarefree}, k_1, k_2 \in \mathbb{N} \}
  $$
  and the complement. The $(n_1, n_2) \in \mathcal{M}$ contribute to a main term that is relatively easy to compute. On the other hand we will bound the contribution of $(n_1, n_2) \not \in \mathcal{M}$.

  \subsubsection{The main term $(n_1, n_2) \in \mathcal{M}$}
  The conditions $d_1^2 n_1 \equiv d_2^2 n_2 \Mod {q}$ and $(n_1 n_2 , q) = $  in the sum in \eqref{eq:second} imply that if $(n_1, n_2) \in \mathcal{M}$ then $d_1^2 k_1^2 \,\equiv\, d_2^2 k_2^2 \Mod{q}$ and therefore $d_1 k_1 \,\equiv\, \pm d_2 k_2 \Mod{q}$. This implies that $d_1 k_1 = d_2 k_2$ since $d_j k_j \leq \sqrt{z} \cdot \sqrt{x^{\varepsilon/3} z q / x} = x^{\varepsilon/6} z \cdot \sqrt{q / x}$ and this is $\leq q/3$ because $z \leq x^{-\varepsilon} \sqrt{qx}$. We conclude that the contribution of $(n_1, n_2) \in \mathcal{M}$ is given by
\begin{equation}\label{eq:prop7_cont_from_M}
\frac{x^2}{q} \sum_{k_1, k_2} \sum_{\substack{\substack{d_1 k_1 = d_2 k_2 \\ d_1^2, d_2^2 \leq z \\ (d_1 d_2 k_1 k_2, q) = 1}}} \frac{\mu(d_1) \mu(d_2)}{d_1^2 \cdot d_2^2} \sum_{(m,q)=1} \mu^2(m) \hat{f} \Big ( \frac{x k_2^2 m}{d_1^2 q} \Big ) \overline{\hat{f} \Big ( \frac{x k_1^2 m}{d_2^2 q} \Big )}.
\end{equation}
We now parametrize the equation $d_1 k_1 = d_2 k_2$ by dividing by $(d_1, d_2)$ on both sides so that
$$
k_1 = \frac{d_2 \ell}{(d_1, d_2)} \text{ and } k_2 = \frac{d_1 \ell}{(d_1, d_2)} \quad \text{with $\ell \in \mathbb{N}$}. 
$$
Plugging this and noticing that each non-negative integer can be written uniquely as $\ell^2 m$ with $m$ squarefree, we can re-write \eqref{eq:prop7_cont_from_M} as
$$
\frac{2 x^2}{q} \sum_{\substack{d_1^2, d_2^2 \leq z \\ (d_1 d_2, q) = 1}} \frac{\mu(d_1) \mu(d_2)}{d_1^2 \cdot d_2^2} \sum_{\substack{\ell \geq 1 \\ (\ell,q) = 1}} \Big | \hat{f} \Big ( \frac{x \ell}{q (d_1^2, d_2^2)} \Big )  \Big |^2. 
$$
Note that we can drop the condition $(d_1d_2,q) = 1$ as $q$ is prime and $d_1, d_2 < q$. Likewise since $\ell \geq q$ contribute $O_A(x^{-A})$, we can drop the condition $(\ell, q) = 1$ and apply Lemma \ref{le:asympt} with $W = \hat{f}$ and $H = x/q$ to see that the above is
$$
C \sqrt{qx} \cdot \pi \int_0^\infty |\hat{f}(y)|^2\sqrt{y}\,dy + O((x/q)^{-\varepsilon/8} \sqrt{xq}).
$$

Let $F(u) = \mathbf{1}_{[0,1]}(u)$. We have,
$$
\hat{f}(y) - \hat{F}(y) \ll \min\{(x/q)^{-\varepsilon/4}, |y|^{-1}\},
$$
with the bound $(x/q)^{-\varepsilon/4}$ for the difference between these two Fourier transforms following from the fact that $\|f - F\|_{L^1} \ll (x/q)^{-\varepsilon/4}$, and the bound $1/|y|$ following from the fact that the total variation of the function $f-F$ is bounded by an absolute constant. Likewise 
$$
\hat{f}(y)\ll (1+|y|)^{-1} \quad \text{and} \quad \hat{F}(y) \ll (1+|y|)^{-1}.
$$

Hence
\begin{multline*}
\int_0^\infty |\hat{f}(y)|^2 \sqrt{y}\, dy - \int_0^\infty |\hat{F}(y)|^2\sqrt{y}\,dy \\ \ll \int_0^\infty \min\{(x/q)^{-\varepsilon/4}, y^{-1}\} (1+y)^{-1} \sqrt{y}\, dy \ll (x/q)^{-\varepsilon/8}.
\end{multline*}
Putting these estimates together, and using the relation $|\hat{F}(\xi)| = |S(\xi)|$ and the integral identity \eqref{eq:sinc_integral}, we see that \eqref{eq:prop7_cont_from_M} is
$$
C\sqrt{qx} + O((x/q)^{-\varepsilon/8} \sqrt{xq}).
$$

\subsubsection{The off-diagonal $(n_1, n_2) \not \in \mathcal{M}$}

Let us focus on bounding the contribution of $(n_1, n_2) \not \in \mathcal{M}$. We recall that the contribution of $d_1 \leq x^{1/2 - \varepsilon/6}/q^{1/2}$ to \eqref{eq:first} is negligible and likewise the contribution of $d_2 \leq x^{1/2 - \varepsilon/6}/q^{1/2}$ is negligible. 
We now partition $d_1, d_2$ into intervals $[D_1, 2D_1]$ and $[D_2, 2 D_2]$ with $x^{1/2 - \varepsilon/6}/q^{1/2} \leq D_1, D_2 \leq \sqrt{z}$.
The total contribution of $(n_1, n_2) \not \in \mathcal{M}$ with $d_1 \in [D_1, 2 D_1]$ and $d_2 \in [D_2, 2 D_2]$ to~\eqref{eq:first} is bounded by \begin{equation} \label{eq:bounding}
\frac{x^2}{q} \cdot  \frac{1}{D_1^2 D_2^2} \sum_{\substack{(n_1, n_2) \not \in \mathcal{M}}} V \Big ( \frac{n_1}{N_1} \Big ) V \Big ( \frac{n_2}{N_2} \Big ) \sum_{\substack{d_1^2 n_1 \,\equiv\, d_2^2 n_2 \Mod{q} \\ (d_1 d_2 n_1 n_2, q) = 1}} V \Big ( \frac{d_1}{D_1} \Big ) V \Big ( \frac{d_2}{D_2} \Big )
\end{equation}
with $V$ a smooth non-negative compactly supported function such that $V(x) \geq 1$ for $x \in [-2, 2]$
and $D_1, D_2 > x^{1/2 - \varepsilon/6}/q^{1/2}$ and $N_1 \leq x^{\varepsilon/3} D_2^2 q / x$ and $N_2 \leq x^{\varepsilon/3} D_1^2 q / x$.

We now split into two cases according to the size of $D_1D_2$:

\subsubsection{Case $D_1 D_2 \geq x^{1 + 2\varepsilon} / q$}

In this case we do not use the condition $(n_1, n_2) \not \in \mathcal{M}$. Dropping this condition and using Dirichlet characters we can re-write \eqref{eq:bounding} as
\begin{align} \label{eq:equat2}
\frac{x^2}{q \varphi(q)} \frac{1}{D_1^2 D_2^2} \sum_{\chi^2 \neq \chi_0} \Big ( \sum_{n_1} & \chi(n_1) V \Big ( \frac{n_1}{N_1} \Big ) \Big ) \Big ( \sum_{n_2} \overline{\chi}(n_2) V \Big ( \frac{n_2}{N_2} \Big ) \Big ) \\ \nonumber & \times \Big ( \sum_{d_1} \chi^2 (d_1) V \Big ( \frac{d_1}{D_1} \Big ) \Big ) \Big ( \sum_{d_2} \overline{\chi}^2(d_2) V \Big ( \frac{d_2}{D_2} \Big ) \Big ) + O \Big ( \frac{x^2}{q^2} \cdot \frac{N_1 N_2}{D_1 D_2} \Big ) 
\end{align}
and where the $O(\cdot)$ term corresponds to the contribution of the characters with $\chi^2 = \chi_0$. Note that this contribution is acceptable since
$$
\frac{x^2}{q^2} \cdot \frac{N_1 N_2}{D_1 D_2} \ll x^{2\varepsilon/3} D_1 D_2 \ll x^{2\varepsilon/3} z \ll x^{-\varepsilon/3} \sqrt{qx}. 
$$
Now we express each of the sums in  \eqref{eq:equat2} using a contour integral, and using H\"older's inequality
this allows us to bound \eqref{eq:equat2} by
$$
\frac{x^2}{q^2} \cdot \frac{\sqrt{N_1 N_2 D_1 D_2}}{D_1^2 D_2^2} \sum_{\chi} \int_{|u| \leq x^{\varepsilon/3}}  |L(\tfrac 12 + i u, \chi)|^4 du + x^{-\varepsilon/3} \sqrt{qx}.
$$
By the fourth moment bound (Lemma~\ref{le:hybridFourth}) the first term is
$$
\ll \frac{x^2}{q^2} \cdot \frac{\sqrt{N_1 N_2 D_1 D_2}}{D_1^2 D_2^2} q x^{\varepsilon/2} \ll x^{5\varepsilon/6} \frac{x}{\sqrt{D_1 D_2}} \ll x^{-\varepsilon/6} \sqrt{qx} 
$$
since $D_1 D_2 \geq x^{1 + 2\varepsilon} / q$. 

\subsubsection{Case $D_1 D_2 < x^{1 + 2\varepsilon} / q$}

In this case we notice that since $D_1, D_2 > x^{1/2 - \varepsilon/6} / \sqrt{q}$ we have $D_1, D_2 \leq (x / q)^{1/2} x^{3\varepsilon}$ and in particular $N_1, N_2 \ll x^{7 \varepsilon}$. We notice that if $(n_1, n_2) \not \in \mathcal{M}$ and $n_1 d_1^2 \,\equiv\, n_2 d_2^2 \Mod{q}$ then $n_1 d_1^2 = n_2 d_2^2 + q \ell$ with $0 < |\ell| \ll x^{1+13\varepsilon}/q^2$.
We now fix $n_1, n_2, \ell$ --- there are $\ll x^{1 + 27\varepsilon} / q^2$ possible choices.
We shall show that the number of solutions in $|d_1|, |d_2| \ll (x / q)^{1/2} x^{3 \varepsilon}$ to $n_1 d_1^2 - n_2 d_2^2 = q \ell$ is bounded by
$\ll x^{9\varepsilon}$ which will be sufficient.

First of all note that we can assume that $(n_1, n_2, q \ell) = 1$. Indeed, $q$ cannot divide $n_1 n_2$ as $n_1 n_2 =o(q)$, and so letting $g=(n_1,n_2, q\ell)$ we have $g \mid \ell$ and the problem reduces to one where $(n_1,n_2,\ell)$ is replaced with $(n_1',n_2',\ell') = (n_1,n_2,\ell)/g$ and now $(n_1',n_2',q\ell')=1$.

Notice that $f(x_1, y_1) = n_1 x_1^2 - n_2 y_1^2$ is a primitive binary quadratic form with discriminant $d = 4 n_1 n_2 > 0$. Denote by $\varepsilon_{n_1 n_2}$ the real number $x_0/2 + y_0 \sqrt{n_1 n_2}$ where $(x_0, y_0)$
is the solution in positive integers to the equation $x_0^2 - 4 n_1 n_2 y_0^2 = 4$ for which $x_0+y_0\sqrt{d}$ is least. Note that $\varepsilon_{n_1 n_2} \geq 3/2$.

Let $(x_1,y_1)$ be a solution to $f(x_1,y_1) = q \ell$ with $x_1 , y_1 \ll (x / q)^{1/2} x^{3\varepsilon}$. We notice that in this situation
$$
(x_1,y_1) \in \bigcup_{\substack{1 \leq m \leq \log x}} T_m^{+} \cup T_m^{-}
$$
where
$$
T_{m}^{+} = \Big \{ (x,y) \in \mathbb{Z}^2 : f(x,y) = q \ell \text{ and } \sqrt{n_1} x > \sqrt{n_2} y \text{ and } \varepsilon_{n_1 n_2}^{2m - 2} \leq \Big | \frac{\sqrt{n_1} x + \sqrt{n_2} y}{\sqrt{n_1} x - \sqrt{n_2} y} \Big | < \varepsilon_{n_1 n_2}^{2m} \Big \}
$$
and
\[
\begin{split}
T_{m}^{-} &= \Big \{ (x,y) \in \mathbb{Z}^2 : f(x,y) = q \ell \text{ and } \sqrt{n_1} x < \sqrt{n_2} y \text{ and } \varepsilon_{n_1 n_2}^{2m - 2} \leq \Big | \frac{\sqrt{n_1} x + \sqrt{n_2} y}{\sqrt{n_1} x - \sqrt{n_2} y} \Big | < \varepsilon_{n_1 n_2}^{2m} \Big \} \\
& = \{(x, y) \in \mathbb{Z}^2 \colon (-x, -y) \in T_m^+\}.
\end{split}
\]
The reason for this is that $\sqrt{n_1} x_1 + \sqrt{n_2} y_1 \ll x^{7\varepsilon} (x/q)^{1/2}$ and
$$
| \sqrt{n_1} x_1 - \sqrt{n_2} y_1 | = \frac{q \ell}{\sqrt{n_1} x_1 + \sqrt{n_2} y_1} \gg \frac{q^{3/2}}{x^{1/2+7\varepsilon}} \gg 1. 
$$
Moreover by Lemma 13 of \cite{Williams1} we have $\# T_{m}^{+} = \# T_{1}^{+}$ for all $m \geq 1$, and trivially $\# T_{m}^{-} = \# T_{m}^{+}$ for all $m \geq 1$. 

The solutions belonging to $T_{1}^{+}$ are \emph{primary} for the quadratic form $n_1 x_1^2 - n_2 x_2^2$ of discriminant
$4 n_1 n_2$ (see p. 101 of \cite{Williams2} for the definition of primary). By Theorem 4.1 of \cite{Williams2}
the number of $(x_1, y_1)$ for which there exists a quadratic form $g$ of discriminant
$4 n_1 n_2$ such that $g(x_1,y_1) = q \ell$ and such that $(x_1, y_1)$ is primary for $g$, is either $0$ or given by
$$
m \prod_{p \mid m} \left( 1-\frac{1}{p}\left( \frac{4n_1 n_2/m^2}{p} \right) \right) \cdot \sum_{k \mid \frac{q\ell}{m^2} } \left( \frac{d_0}{k}\right),
$$
for particular integers $m$ and $d_0$ with $m^2 \mid (q\ell,4n_1 n_2)$. Using the divisor bound $\# \{ k : k \mid n\} \ll_{\varepsilon} n^{\varepsilon/100}$, we find that this is 
$$
\ll (n_1 n_2)^{1/2 + \varepsilon/100} (q \ell)^{\varepsilon/100} \ll x^{8\varepsilon}.
$$
We conclude therefore that $\# T_{1}^{+} \ll x^{8\varepsilon}$ and therefore the number of solutions $(x_1,y_1)$ with $|x_1|, |y_1| \ll ( x/q)^{1/2} x^{3\varepsilon}$ to
the equation $f(x_1, y_1) = q \ell$ is bounded by $\ll \log x\cdot\# T_{1}^{+} \ll x^{9\varepsilon}$
as claimed. It follows therefore that the total number of solutions to
$n_1 d_1^2 - n_2 d_2^2 = q \ell$ with $n_i \sim N_i$, $d_i \sim D_i$ for $i = 1,2$
is $\ll x^{1 + 36\varepsilon} / q^2$. 

We conclude therefore that \eqref{eq:bounding} is
$$
\ll \frac{x^2}{q} \cdot \frac{1}{D_1^2 D_2^2} \cdot \frac{x^{1 + 36\varepsilon}}{q^2} \ll \frac{x^{1 + 40 \varepsilon}}{q} \ll x^{-\varepsilon} \sqrt{ qx}
$$
since $q > x^{1/3 + 30\varepsilon}$. 

    \section{The range $z \ge (x / q)^{4/3 + \varepsilon}$ in the $q$-aspect : Proof of Proposition \ref{prop:prop2q}}
\label{se:prop2q}
Splitting into dyadic segments and recalling~\eqref{eq:transWOchi0}, we can bound the left-hand side of~\eqref{eq:prop2qclaim} by a constant times
    $$
    \log x \sup_{\sqrt{z} \leq D \leq \sqrt{x}} \frac{1}{\varphi(q)^2} \sum_{\chi \neq \chi_0} \Big | \sum_{\substack{n d^2 \leq x \\ d \sim D}} \mu(d) \chi(n) \chi(d^2) \Big |^2.
    $$
    Expressing the condition $n d^2 \leq x$ using a contour integral (see \cite[Cor. 5.3]{montgomery2007}) the above is bounded by
    $$
    \ll \log x \sup_{\sqrt{z} \leq D \leq \sqrt{x}} \frac{1}{\varphi(q)^2} \sum_{\chi \neq \chi_0} \Big | \int_{|t| \leq x} L(\tfrac 12 + it, \chi) M(1 + 2 i t, \chi^2) \cdot \frac{x^{1/2 + it}}{1/2 + it} dt \Big |^2  + O((x/q)^{1/2-\varepsilon/8}),
    $$
(in fact a better error term can be obtained but we do not need to keep track of it) where
    $$
    M(1  + 2 i t, \chi^2) = \sum_{d \sim D} \frac{\mu(d) \chi^2(d)}{d^{1 + 2 it}}. 
    $$
    Applying Cauchy-Schwarz and splitting according to the values of $t$ we can bound the main term above as
    \begin{equation}
	\label{eq:dyadicLupperBound}
    \ll x (\log x)^3 \sup_{\substack{\sqrt{z} \leq D \leq \sqrt{x} \\ 1 \leq T \leq x}} \frac{1}{\varphi(q)^2} \sum_{\chi \neq \chi_0} \frac{1}{T} \int_{-T}^{T} |L(\tfrac 12 + it, \chi)|^2 \cdot |M(1 + 2it, \chi^2)|^2 dt.
    \end{equation}

Let us now prove the claim on the assumption of the Generalized Lindel\"of Hypothesis. Applying Generalized Lindel\"of and then the hybrid mean-value theorem (Lemma \ref{le:hmvt}) we have for any choice of $\delta > 0$,
\begin{equation}
\label{eq:qTintest}
\begin{split}
&\frac{x}{\varphi(q)^2} \sum_{\chi \neq \chi_0} \frac{1}{T} \int_{-T}^T |L(\tfrac{1}{2}+it,\chi)|^2 |M(1+2it,\chi^2)|^2\, dt \\
&\ll \frac{x (qT)^\delta}{q^2 T} \sum_{\chi \neq \chi_0} \int_{-T}^T |M(1+2it,\chi^2)|^2\,dt \\
&\ll \frac{x T^\delta q^{2\delta}}{q^2 T} (qT + D) \cdot \frac{1}{D} \ll T^\delta q^{2\delta} \left(\frac{x}{qD} +  \frac{x}{q^2 T}\right).
\end{split}
\end{equation}
Since $q, T \leq x \leq (x/q)^{O(1)}$, for sufficently small $\delta$ we have $T^\delta q^{2\delta} \leq (x/q)^{\varepsilon/100}$. Recalling also that $D \geq z^{1/2} \geq (x/q)^{(1+\varepsilon)/2}$ and $q \geq x^{1/3 + \varepsilon}$, we see that~\eqref{eq:qTintest} is
$$
\ll x^{\varepsilon/100} \biggl( \Bigl(\frac{x}{q}\Bigr)^{1/2-\varepsilon/2} + \frac{x}{q^{2}} \biggr) \ll (x/q)^{1/2-\varepsilon/3}.
$$
Applying this estimate to \eqref{eq:dyadicLupperBound} yields the claim.

  Let us now consider the unconditional part of the claim. 
    Let
    $$
    S_{T,q} (V) := \{ (t, \chi) : V \leq | M(1 + 2 it, \chi^2)| \leq 2 V \ , \ |t| \leq T \ , \ \chi \Mod{q} \}. 
    $$
    Note that for $D \geq \sqrt{z} \geq (x/q)^{1/2+\varepsilon}$, the values of $t \in [-T, T]$ for
    which $|M(1+2it, \chi^2)| \leq D^{-1/2 + \varepsilon/4}$ contribute to ~\eqref{eq:dyadicLupperBound} by Cauchy-Schwarz and the fourth moment bound (Lemma~\ref{le:hybridFourth})
    $O((\log x)^5 x D^{-1 + \varepsilon/2}/q) = O((x / q)^{1/2 - \varepsilon/2})$. Additionally, $|M(1+2it)| \le \sum_{d \sim D} 1/d \le 2$. Therefore it suffices to show that for each $\sqrt{x} \geq D \geq \sqrt{z}$, $V \in [D^{-1/2}, 1]$, and $T \in [1, x]$, we have
    \begin{equation}
    	\label{eq:required ap bnd}
     \frac{x V^2}{\varphi(q)} \sum_{\chi \neq \chi_0} \frac{1}{T} \int_{t \colon (t, \chi) \in S_{T, q}(V)} |L(\tfrac 12 + it, \chi)|^2 dt \ll (x/q)^{-\varepsilon/8} (\log x)^{-4}  \cdot \sqrt{q x}.
    \end{equation}
     By Lemma \ref{le:hybridLVT} we have,
     \begin{equation}
     \label{eq:STqbound}
     |S_{T, q}(V)| \ll (V^{-2} + q T \min \{ D^{-1} V^{-2}, D^{-2} V^{-6} \}) \cdot (\log x)^{18}.
     \end{equation}
     Here $|S_{T,q}(V)|$ is the measure of $S_{T,q}(V)$, where the set of $\chi \Mod{q}$ is endowed with the counting measure.
     Consider first the case when the first term dominates. Then, by Lemma \ref{le:hybrid}, we see that the left-hand side of \eqref{eq:required ap bnd} is
     \begin{equation*}
     \ll \frac{x V^2}{\varphi(q)} \cdot \frac{|S_{T, q}(V)|}{T} \cdot (q T)^{1/3 + \varepsilon/4} \ll \frac{x}{\varphi(q)} \cdot \frac{1}{T} \cdot (q T)^{1/3 + \varepsilon/3} \ll \frac{x}{q^{2/3 - \varepsilon/2}} \ll (x/q)^{-\varepsilon/8} (\log x)^{-4}  \cdot \sqrt{qx}
     \end{equation*}
     since $q > x^{3/7 + \varepsilon}$. Note that the factor $(\log x)^{18}$ in \eqref{eq:STqbound} was absorbed in the exponent of $qT$.

     Consider now the case that the second term dominates in~\eqref{eq:STqbound}. Then by Cauchy-Schwarz and the hybrid fourth moment estimate (Lemma~\ref{le:hybridFourth}),
     \begin{align*}
     \frac{x V^2}{\varphi(q)} & \sum_{\chi \neq \chi_0} \frac{1}{T} \int_{t \colon (t, \chi) \in S_{T, q}(V)} |L(\tfrac 12 + it, \chi)|^2 dt \\ & \ll \frac{x V^2}{T \varphi(q)} \cdot |S_{T, q}(V)|^{1/2} \cdot \Big ( \sum_{\chi \neq \chi_0} \int_{-T}^T |L(\tfrac 12 + it, \chi)|^4 dt \Big )^{1/2} \\
                                                                                                                                   & \ll (\log x)^{11} \cdot x V^2 \cdot \min \{ D^{-1} V^{-2}, D^{-2} V^{-6} \}^{1/2} \\
& \ll (\log x)^{11} \cdot x \min \{ D^{-1/2} V, D^{-1} V^{-1} \} \\                                                                                                                                   
       & \ll (\log x)^{11} \cdot x \cdot ( D^{-1/2} V )^{1/2} \cdot (D^{-1} V^{-1})^{1/2} \\ & \ll x (\log x)^{11} \cdot D^{-3 / 4} \ll  (x/q)^{-\varepsilon/8} (\log x)^{-3}  \sqrt{qx} 
     \end{align*}
     since $D \geq \sqrt{z} \geq (x / q)^{2/3 + \varepsilon/2}$. 

\section{Conditional estimates: Proof of Theorem \ref{thm:main_upper}}
\label{sec:GRH_bounds}

The proof of Theorem \ref{thm:main_upper} splits into two parts since two assertions are made.
\subsection{Proof that the Riemann Hypothesis implies \eqref{eq:main_upper}}

The proof follows the same ideas as the proof of Proposition \ref{pr:prop2}. The claim \eqref{eq:main_upper} is already proved for $H\leq X^{2/3-\varepsilon}$, so we may assume $H > X^{2/3-\varepsilon}$. We return to \eqref{eq:final} and consider first the case $D \geq H^{(1-\delta)/2}$. Note that the Riemann Hypothesis implies
\begin{equation}
\label{eq:mobiusRHbound}
M(1+2it) \ll_\delta D^{-1/2+\delta/2},
\end{equation}
for $|t| \leq X^2$. Now \eqref{eq:final}, Cauchy-Schwarz and the fourth moment bound for the Riemann zeta function (Lemma~\ref{le:ZetaFourth}) imply
$$
\frac{1}{X} \int_{X}^{2X} \Big | \sum_{\substack{x < n d^2 \leq x + H \\ d \sim D}} \mu(d) - H \sum_{\substack{d \sim D}} \frac{\mu(d)}{d^2} \Big |^2 dx \ll_\delta (\log X)^2 H/D^{1-\delta}.
$$
For $D \geq H^{(1-\delta)/2}$, the right-hand side is $\ll_{\delta} (\log X) H^{1/2+\delta - \delta^2/2}$.
Splitting dyadically for $D \in [H^{(1-\delta)/2}, X^{1/2}]$ and using the tail bound
$$
\sum_{d^2 > 2X} \frac{\mu(d)}{d^2} \ll_{\delta} \frac{1}{X^{3/4 - \delta/10}},
$$
valid under Riemann Hypothesis, 
we see that the Riemann Hypothesis implies
\begin{equation}
\label{eq:RH_upper_est}
\begin{split}
&\frac{1}{X} \int_{X}^{2X} \Big | \sum_{\substack{x < n d^2 \leq x + H \\ d^2 \geq H^{1-\delta} }} \mu(d) - H \sum_{\substack{d^2 \geq H^{1-\delta} }} \frac{\mu(d)}{d^2} \Big |^2 dx \\
& \ll (\log X)^2 \sup_{H^{(1-\delta)/2} \leq D \leq X^{1/2}} \frac{1}{X} \int_{X}^{2X} \Big | \sum_{\substack{x < n d^2 \leq x + H \\ d \sim D }} \mu(d) - H \sum_{\substack{d \sim D}} \frac{\mu(d)}{d^2}\Big|^2 dx \\
& \qquad \qquad \qquad \qquad \qquad + \frac{1}{X} \int_{X}^{2X} \Big |H \sum_{d^2 > 2X} \frac{\mu(d)}{d^2} \Big|^2 dx \ll_\delta H^{1/2+\delta}.
\end{split}
\end{equation}

On the other hand, estimating the $n$-sum on the left-hand side by $H/d^2 + O(1)$, we see that
\[
\frac{1}{X} \int_{X}^{2X} \Big | \sum_{\substack{x < n d^2 \leq x + H \\ d^2 \leq H^{1/2} }} \mu(d) - H \sum_{\substack{d^2 \leq H^{1/2} }} \frac{\mu(d)}{d^2} \Big |^2 dx \ll H^{1/2}.
\]

Hence the claim follows once we have shown that, for any $D \in [H^{1/4}, H^{(1-\delta)/2}]$, we have
\[
\frac{1}{X} \int_{X}^{2X} \Big | \sum_{\substack{x < n d^2 \leq x + H \\ d \sim D}} \mu(d) - H \sum_{\substack{d \sim D}} \frac{\mu(d)}{d^2} \Big |^2 dx \ll_{\delta} H^{1/2}.
\]
Notice that we can attach to the $n$ variable a dummy function $f(n D^2 / X)$ with $f$ a smooth function supported in $[1/20, 20]$ and such that $f(y) = 1$ for $y \in [1/10, 10]$. 

Similarly to the proof of Propositon~\ref{pr:prop2}, write
$$
A(x) := \sum_{\substack{n d^2 \leq x \\ d \sim D}} f \Big ( \frac{n}{X / D^2} \Big ) \mu(d) - x \sum_{d \sim D} \frac{\mu(d)}{d^2}.
$$
By contour integration we have,  for $e^y \in [X, 2X]$ and $w \leq 1/100$, 
\begin{equation}
\label{eq:Aey2}
A(e^{y + w}) - A(e^{y}) = \frac{1}{2\pi i} \int_{1/2-i\infty}^{1/2+i\infty} e^{y s} \cdot \frac{e^{w s} - 1}{s} N_1(s) M(2s) ds - e^{y} (e^{w} - 1) \sum_{d \sim D} \frac{\mu(d)}{d^2},
\end{equation}
where
$$
M(s) := \sum_{d \sim D} \frac{\mu(d)}{d^{s}} \quad \text{and} \quad N_1(s) := \sum_{m} \frac{1}{m^s} \cdot f \Big ( \frac{m}{X / D^2} \Big ).
$$

Write also
\[
N_2(s) := \int_{\mathbb{R}} \frac{1}{u^s} \cdot f \Big ( \frac{u}{X / D^2} \Big ) du 
\]
and note that, for $e^y \in [X, 2X]$, we have by contour integration
\begin{align*}
\frac{1}{2\pi i} \int_{1/2-i\infty}^{1/2+i\infty} e^{y s} \cdot \frac{e^{w s} - 1}{s} N_2(s) M(2s) ds & = \sum_{d \sim D} \mu(d) \int_{\mathbb{R}} f \Big ( \frac{u}{X / D^2} \Big ) 1_{e^{y} \leq ud^2 \leq e^{y + w}} du \\ & = e^y (e^{w} - 1) \sum_{d \sim D} \frac{\mu(d)}{d^2}. 
\end{align*}
Plugging this into~\eqref{eq:Aey2} and arguing as in the proof of Proposition~\ref{pr:prop2}, we see that, for some $w \asymp H/X$, we have
\begin{equation}
\label{eq:AN1-N2bound}
\begin{split}
&\frac{1}{X} \int_{X}^{2X} |A(x + H) - A(x)|^2 dx \\
&\ll X \int_{\mathbb{R}} \Big | \frac{e^{w (\tfrac 12 + it)} - 1}{\tfrac 12 + it} \Big |^2 \cdot \left|\left(N_1(\tfrac 12 + it)-N_2(\tfrac{1}{2}+it)\right) M(1 + 2it)\right|^2 dt.
\end{split}
\end{equation}
By Poisson summation
\[
\begin{split}
N_1(\tfrac 12 + it) &= \sum_{m} \frac{1}{m^{1/2+it}} \cdot f \Big ( \frac{m}{X / D^2} \Big ) = \sum_\ell \int_{-\infty}^\infty \frac{1}{u^{1/2+it}} \cdot f \Big ( \frac{u}{X / D^2} \Big ) e(\ell u) du \\
& = N_2(\tfrac{1}{2} + it) + \frac{X}{D^2} \sum_{\ell \neq 0} \int_{1/20}^{20} \left(\frac{D^2}{yX}\right)^{1/2+it} f( y ) e\left(\frac{\ell yX}{D^2}\right) dy. 
\end{split}
\]
By partial integration (taking antiderivatives of $e(\ell y X/D^2)$), this implies that, for $|t| < X/D^{2 + \delta/100}$, 
\[
|N_1(\tfrac 12 + it)-N_2(\tfrac{1}{2}+it)| \ll_{A} X^{-A}, 
\]
for any $A > 0$. Therefore the part of the integral~\eqref{eq:AN1-N2bound} with $|t| < X / D^{2 +\delta/100}$ is completely negligible. 

On the other hand the part with $|t| \geq X^{10}$ contributes only $O(1)$ to the left-hand side of~\eqref{eq:AN1-N2bound} by estimating $|N_j(1/2+it)|$ and $|M(1+it)|$ trivially.

Furthermore, assuming the Riemann Hypothesis, we have by contour integration, for $|t| \in [X/D^{2 + \delta/100}, X^{10}]$, 
\[
|N_1(\tfrac 12 + it)| \ll \sup_{\substack{|t|/2 \leq |u| \leq 2X^{10}}} |\zeta(\tfrac{1}{2}+iu)| \ll_{\delta} X^{\delta/100}
\]
and 
\[
|M(1+2it)| \ll_{\delta} D^{-1/2+\delta/100}.
\] 
Furthermore, by partial integration we have, for $|t| \in [X/D^{2 + \delta/100}, X^{10}]$, 
\[
|N_2(\tfrac 12 + it)| \ll_{\delta} X^{\delta/100}
\]
Hence the part with $|t| \in [X/D^{2 + \delta/100}, X^{10}]$ contributes to~\eqref{eq:AN1-N2bound}
\[
\ll_{\delta} X \int_{X/D^{2 + \delta/100} \leq |t| \leq X^{10}} \frac{1}{|t|^2} X^{\delta/50} D^{-1+\delta/50}\, dt  \ll_{\delta} D X^{\delta/10} \ll_{\delta} H^{1/2}
\]
since $D \leq H^{(1-\delta)/2}$ and $H \geq X^{1/2}$.

\subsection{Proof that \eqref{eq:main_upper} implies the Riemann Hypothesis}

Suppose that \eqref{eq:main_upper} holds for $H = X^{1 - \delta}$. 
Then, by Cauchy-Schwarz, 
$$
\int_{\mathbb{R}} \Phi \Big ( \frac{x}{X} \Big ) \Big ( \frac{1}{H} \sum_{x < m \leq x + H} \mu^2 (m) \Big ) dx = \frac{6 X\widehat{\Phi}(0)}{\pi^2} + O_{\delta}(X^{1/4 + 3 \delta}). 
$$
with $\Phi$ an arbitrary, but not identically zero smooth function compactly supported in $[1/2, 3]$ (one could even enforce that $\widehat{\Phi}(0) = 0$ to simplify the above expression but we didn't find any significant advantage in doing this). 
Therefore,
\begin{equation} \label{eq:beginn}
\frac{1}{2\pi i} \int_{2 - i \infty}^{2 + i \infty} \frac{\zeta(s)}{\zeta(2s)} X^s \cdot \Psi_{H/ X}(s) ds - \frac{6 X \widehat{\Phi}(0)}{\pi^2} = O_{\delta}(X^{1/4 + 3 \delta}),
\end{equation}
where uniformly in $1/100 < \Re s < 100$, for any given $A > 1$, 
\begin{equation}
\begin{aligned} \label{eq:repl}
  \Psi_{H/X}(s) & := \frac{1}{s} \cdot \frac{X}{H} \int_{\mathbb{R}} \Big ( \Phi \Big (x - \frac{H}{X} \Big ) - \Phi(x) \Big ) x^s dx\\ & = \frac{1}{s} \sum_{1 \leq j \leq A} \frac{(-1)^j}{j!} \cdot \Big ( \frac{H}{X} \Big )^{j - 1} \int_{\mathbb{R}} \Phi^{(j)}(x) x^{s} dx + O_{A} ( X^{-\delta A} )  \\
  & = - \frac{1}{s} \int_{\mathbb{R}} \Phi'(x) x^{s} dx + O_{A} \Big ( \frac{H}{X} \cdot (1 + |\Im s|)^{-A} + X^{-\delta A} \Big ).
\end{aligned}
\end{equation}
By integration by parts the main term is equal to $\widetilde{\Phi}(s)$ where $\widetilde{\Phi}(s)$ is the Mellin transform of $\Phi$. The reader may also verify that we have the exact relation $\Psi_{H/X}(1) = \widetilde{\Phi}(1)$.

Suppose that the Riemann Hypothesis fails. Then $\zeta(s) / \zeta(2s)$ has a pole in the strip $\tfrac 14 < \sigma < \tfrac 12$ (e.g. $s = \rho / 2$ with $\rho = \beta + i \gamma$ the zeros of $\zeta(s)$ with smallest $\gamma > 0$ among all zeros of $\zeta(s)$ with $\beta \in (\tfrac 12, 1)$).
Let $\Theta > \tfrac{1}{4}$ denote the supremum of the real part of poles of $\zeta(s) / \zeta(2s)$ lying in the strip $\tfrac 14 < \sigma < \tfrac 12$. 
Choose $\delta > 0$ to be sufficiently small so that $\tfrac{1}{4} + 3 \delta \leq \Theta - \delta / 2$. 

Pick now $s_0$ a pole of $\zeta(s) / \zeta(2s)$ with $\Re s_0 \in (\Theta - \delta / 50, \Theta]$ and the smallest positive imaginary part. We can assume without loss of generality that $\Phi$ is chosen so that $\widetilde{\Phi}(s_{0}) \neq 0$. Indeed if it were the case that $\widetilde{\Phi}(s_0) = 0$ then pick a $c \in (0,1)$ such that $\widetilde{\Phi}(c + s_0) \neq 0$ and consider $x^{c} \Phi(x)$ in place of $\Phi(x)$.

We shall shift the contour in \eqref{eq:beginn} to the line $\sigma = \Theta + \delta / 8$. Note that for any fixed values of $X$ and $H$, from the definition \eqref{eq:repl} and integration by parts, we have $\Psi_{H/X}(s) \ll_A (1+|\Im s|)^{-A}$ uniformly for $1/100 < \Re s < 100$ for all $A \geq 1$. Furthermore for $\Re s \in [\sigma,2]$ and $s$ bounded away from $1$ we get the bound $\zeta(s)/\zeta(2s) \ll_\delta (1+|\Im s|)^C$, where $C$ is a constant which depends only on $\delta$. (This follows because we may bound $\zeta(s)$ using a convexity bound (see e.g.~\cite[Sec. 5.1]{Titchmarsh86}) and we may bound $1/\zeta(2s)$ using the estimate $\log \zeta(2s) \ll_\delta \log(|\Im(2s)|+2)$ in this region, which follows from a well-known estimate on the logarithmic derivative of the zeta function (e.g.~\cite[Thm 9.6 (A)]{Titchmarsh86}) and the fact that in this region $|2s-\rho'| \geq \delta/8$ whenever $\zeta(\rho') = 0$.) Thus
$$
\frac{1}{2\pi i} \int_{(2)} \frac{\zeta(s)}{\zeta(2s)} X^s \cdot \Psi_{H/ X}(s) ds = \frac{1}{2\pi i} \int_{(\sigma)} \frac{\zeta(s)}{\zeta(2s)} X^s \cdot \Psi_{H/ X}(s) ds + \frac{6}{\pi^2} X \Psi_{H/X}(1).
$$
Applying \eqref{eq:beginn} and \eqref{eq:repl} we find
$$
\frac{1}{2\pi i} \int_{(\sigma)} \frac{\zeta(s)}{\zeta(2s)} \cdot X^s \widetilde{\Phi}(s) ds = O_{\delta} ( X^{\Theta - \delta / 2} ) + O_{\delta}(X^{1/4 + 3 \delta}).
$$
By choice of $\delta > 0$ the error term is bounded by $O_{\delta}(X^{\Theta - \delta / 2})$. Therefore setting
$$
A(X) := \sum_{n} \mu^2(n) \Phi \Big ( \frac{n}{X} \Big ) - \frac{6}{\pi^2} \widehat{\Phi}(0) X\cdot \mathbf{1}_{[1,\infty)}(X),
$$
we have for $X \geq 1$,
$$
A(X) = \frac{1}{2\pi i} \int_{(\sigma)} \frac{\zeta(s)}{\zeta(2s)} \cdot X^s \widetilde{\Phi}(s) ds = O_{\delta}(X^{\Theta - \delta / 2}).
$$
Thus there exists a constant $c = c(\Theta, \delta)$ such that,
\begin{equation} \label{eq:assumed}
|A(x)| \leq c x^{\Theta - \delta / 50}
\end{equation}
for all $x \geq 0$ (note that for $0 < x < 1/100$ we have that $A(x)$ vanishes). 
Let us start by observing that for $\Re s > 1$,
\begin{equation}  
  \begin{aligned}\label{eq:analyticcontinuation}
    \int_{0}^{\infty} A(x) x^{-s - 1} dx & = \sum_{n \geq 1} \mu^2(n) \int_{0}^{\infty} \Phi \Big ( \frac{n}{x} \Big ) x^{-s - 1} dx - \frac{6 \widehat{\Phi}(0)}{\pi^2} \cdot \frac{1}{s - 1} \\
    & = \sum_{n \geq 1} \mu^2(n) \cdot n^{-s} \widetilde{\Phi}(s) - \frac{6 \widetilde{\Phi}(1)}{\pi^2} \cdot \frac{1}{s - 1} \\ & = \frac{\zeta(s)}{\zeta(2s)} \cdot \widetilde{\Phi}(s) - \frac{6 \widetilde{\Phi}(1)}{\pi^2} \cdot \frac{1}{s - 1}. 
\end{aligned}
\end{equation}
The function $\int_{0}^{\infty} A(x) x^{-s - 1} dx$ is analytic in $\Re s > \Theta - \delta / 50$ by \eqref{eq:assumed}. Therefore, by \eqref{eq:analyticcontinuation} and analytic continuation, 
\begin{equation} \label{eq:righthandside}
\frac{\zeta(s)}{\zeta(2s)} \cdot \widetilde{\Phi}(s) - \frac{6 \widetilde{\Phi}(1)}{\pi^2} \cdot \frac{1}{s - 1}
\end{equation}
is analytic in the region $\Re s > \Theta - \delta / 50$. This however contradicts that \eqref{eq:righthandside} has a pole at $s_0$ and $\Re s_0 \in (\Theta - \delta / 50, \Theta]$.
\bibliographystyle{alpha}
\bibliography{references}

\Addresses

\end{document}